\documentclass[12pt,reqno]{amsart}
\makeatletter
\@namedef{subjclassname@2020}{%
	\textup{2020} Mathematics Subject Classification}
\makeatother
\usepackage{amssymb, a4wide}
\usepackage{lmodern}
\usepackage{comment}
\usepackage{verbatim}
\usepackage{orcidlink}
\usepackage[utf8]{inputenc} 
\usepackage[T1]{fontenc}    
\usepackage{url}            
\usepackage{booktabs}       
\usepackage{amsfonts}       
\usepackage{nicefrac}       
\usepackage{microtype}      
\usepackage{color}
\usepackage{lipsum}
\usepackage{float}
\usepackage{calrsfs}
\usepackage{enumerate}
\usepackage{tikz-cd}
\usepackage{tikz}
\usepackage{graphicx}
\usepackage{comment}
\DeclareMathAlphabet{\pazocal}{OMS}{zplm}{m}{n}

\usepackage{amsfonts}
\usepackage{ stmaryrd }
\usepackage{eqnarray,amsmath}
\usepackage{marginnote}
\usepackage{mathtools}
\usepackage{amsbsy}
\usepackage{hyperref}
\usepackage{arydshln}

\input xygraph
\input xy
\xyoption{all}

\usepackage{makecell,multirow}
\usepackage{rotating}
\newcolumntype{C}[1]{>{\centering\let\newline\\\arraybackslash\hspace{0pt}}m{#1}}
\numberwithin{equation}{section}

\def\N{\mathbb N}

\DeclareMathOperator{\spa}{span}

\def\remove#1{}

\newtheorem{lemma}{Lemma}[section]
\newtheorem{corollary}[lemma]{Corollary}
\newtheorem{theorem}[lemma]{Theorem}
\newtheorem{proposition}[lemma]{Proposition}

\theoremstyle{definition}
\newtheorem{construction}[lemma]{Construction}
\newtheorem{example}[lemma]{Example}

\theoremstyle{remark}
\newtheorem{remark}[lemma]{Remark}
\newtheorem{interp}[lemma]{Biological interpretation}

\usepackage{array}
\newcolumntype{R}[1]{>{\arraybackslash$}p{#1}<{$}}
\title[Gonosomal algebras and operators]{ 
	Gonosomal algebras and operators associated to genetic systems with a single male genotype
}
\author[Y. Cabrera]{Yolanda Cabrera Casado\textsuperscript{1}\,\orcidlink{0000-0003-4299-4392}}
\address{\textsuperscript{1}Departamento de Matem\'atica Aplicada, E.T.S. Ingenier\'\i a Inform\'atica, Universidad de M\'alaga, Campus de Teatinos s/n, 29071 M\'alaga,   Spain.}
\email{yolandacc@uma.es}

\author[M. Ladra]{Manuel Ladra\textsuperscript{2}\,\orcidlink{0000-0002-0543-4508}}
\address{\textsuperscript{2}Departamento de Matemáticas \& CITMAga, Universidade de Santiago de Compostela, 15782 Santiago de Compostela, Spain}
\email{manuel.ladra@usc.es}

\author[A. Pérez-Rodríguez]{Andrés Pérez-Rodríguez\textsuperscript{3}\,\orcidlink{0009-0007-1095-5328}}
\address{\textsuperscript{3}Departamento de Matemáticas \& CITMAga, Universidade de Santiago de Compostela, 15782 Santiago de Compostela, Spain}
\email{andresperez.rodriguez@usc.es}

\begin{document}
	\begin{abstract}
		This article is devoted to studying gonosomal algebras and operators with a single male genotype. We compute the limit points of the trajectories of the corresponding normalised gonosomal operators, describing the development of specific populations and providing the corresponding biological interpretations.
	\end{abstract}
	\maketitle
	\section{Introduction}
A population is defined as a group of individuals of the same species that inhabit a specific geographic region at a given time and have the ability to interbreed. The analysis of populations, along with the factors and mechanisms that regulate them, is essential for understanding ecosystems. This task is undertaken by population dynamics, historically recognised as one of the most prominent branches of mathematical biology (see \cite{B_11_popdyn}).

Etherington introduced abstract algebra into the study of genetics in his series of papers~\cite{Etherington_41_Dup,Etherington_40,Etherington_41}.
He defined several classes of non-associative structures, such as baric or train algebras, which have significantly contributed to population genetics. Over time, numerous studies (see \cite{Lyubich_92,Reed_97,Tian_08,TV_06,WB_80}, for example) have confirmed that non-associative algebras are the appropriate mathematical framework for studying genetics. As a result, the term ``genetic algebras'' was coined to refer to these primarily non-associative  algebras, which enable the modelling of inheritance in the field of genetics.

In this paper, we focus on the dynamics of certain bisexual populations. One of the main challenges when constructing algebraic models for sex-linked inheritance in such populations lies in the diversity of sex-determination systems, which are the biological mechanisms that determine the development of sexual characteristics in organisms. Although sex can be determined by environmental factors (such as temperature \cite{V_19_temperature}), it is primarily controlled by a pair of chromosomes known as gonosomes. The most widely known sex-determination system is the male heterogametic XX/XY system found in most mammals in which XY individuals develop as males and XX individuals develop as females. Other organisms, such as birds, have a ZZ/ZW chromosomal system in which heterogametic ZW individuals develop as females and homogametic ZZ individuals develop as males. In fact, there even exist polygenic sex determination systems (see \cite{MR_13}), where the sex of an organism is influenced by multiple genes rather than being determined by a single gene or chromosome.

Ladra and Rozikov established the foundation for algebraically modelling the evolution of bisexual populations in \cite{LR_13} by introducing the evolution algebra of a bisexual population (EABP). However, certain sex-determination systems, such as haemophilia (see \cite[Example~1]{V_16}), cannot be accurately modelled using an EABP. To address this limitation, in \cite{V_16}, Varro extends the definition of EABP by introducing gonosomal algebras. Furthermore, he illustrates several constructions of gonosomal algebras with nearly twenty genetic examples, demonstrating their ability to algebraically represent a wide range of genetic phenomena related to sex.

Every gonosomal algebra gives rise to a quadratic operator known as gonosomal operator, which connects the genetic states of two successive generations. However, this operator does not map the simplex into itself (as illustrated, for instance, in the case of haemophilia in \cite[Lemma~1]{RV_15}), complicating the biological interpretation of its dynamics. To deal with this issue, the normalised gonosomal operator is introduced, which, as we will see later, ensures that the simplex is mapped into itself. This normalised operator links the frequency distributions of genetic traits across consecutive generations. A key challenge lies in determining the limit points of the trajectories generated by these operators from an arbitrary initial point. While there are several studies exploring the dynamical behaviour of specific operators (see \cite{A_19_asymp,A_21_attract,AR_20,A_20_eigen,RSV_24,RV_15}, for example), the absence of a general method for analysing non-linear discrete dynamical systems further complicates the problem, making it a particularly difficult and open area of research.

In \cite{V_16}, examples of sex-determination systems with a single male genotype are presented. 

Another example, involving finitely many female genotypes and only one male genotype, is analysed in detail from an algebraic point of view in \cite{LLR_14}. Motivated by these cases, this paper focuses on studying gonosomal algebras and operators that model these particular genetic systems. We compute the limit points of the trajectories of the normalised gonosomal operators, illustrating the development of specific populations and deriving the corresponding biological interpretations.

The text is structured into five sections. Following this introduction, Section \ref{prelim} covers the preliminaries, reviewing the basic language and the fundamental concepts of gonosomal algebras and (normalised) gonosomal operators within the context of genetic systems with a single male genotype. 
Subsequently, Sections \ref{wolbachia} and \ref{lemming} build upon such genetic examples given in \cite{V_16}. In particular, in Section \ref{wolbachia} we explore a ZW sex-determination system in which infection by the bacterium \textit{Wolbachia} induces ZZ individuals to develop as females, and we analyse its development in terms of Wolbachia's transmission rate to the offspring. Section~\ref{lemming} is devoted to the study of XY sex-determination systems that model the population of some rodents, such as the African pygmy mouse or the Arctic lemming. Notably, in both sections, we have been pioneers in restricting the domain of certain normalised gonosomal operators to obtain meaningful biological interpretations. Finally, in Section \ref{pez}, we provide the first complete mathematical model of African cichlid fish populations, addressing a case that had not been previously explored in the literature.

	\section{Preliminaries on gonosomal algebras and operators}\label{prelim}
	
	Throughout this paper, we denote the set of natural numbers by $\N=\{0,1,\ldots\}$ and $\N^*:=\N\setminus \{0\}$. Moreover, we denote $\Lambda:=\{1, \ldots,n\}$.
    Given a  field $\mathbb{K}$ with characteristic different from $2$, a $\mathbb{K}$-algebra $\mathcal{A}$ is \textit{gonosomal} of type $(n,m)$ if it admits a basis $(e_i)_{1\leq i\leq n}\cup(\widetilde{e}_p)_{1\leq p\leq m}$ such that for all $1\leq i,j\leq n$ and $1\leq p,q\leq m$ we have that
	\[
	e_ie_j=0,\quad \widetilde{e}_p\widetilde{e}_q=0 \quad\text{and}\quad e_{i}\widetilde{e}_{p}=\widetilde{e}_{p}e_{i}= \sum_{k=1}^{n}\gamma_{ipk}e_{k}+\sum_{r=1}^{m}\widetilde{\gamma}_{ipr}\widetilde{e}_{r},\]
	where $\sum_{k=1}^{n}\gamma_{ipk}+\sum_{r=1}^{m}\widetilde{\gamma}_{ipr}=1$. The basis $(e_i)_{1\leq i\leq n}\cup(\widetilde{e}_p)_{1\leq p\leq n}$ is called a \textit{gonosomal basis} of $\mathcal{A}$. All gonosomal algebras considered in this work will be \textit{stochastic}, that is $K=\mathbb{R}$ and $\gamma_{ipk},\widetilde{\gamma}_{ipr} \geq 0$, and will model sex-determination genetic systems with a single male genotype. Hence, from now on, we will assume that all the gonosomal algebras are $\mathbb{R}$-algebras which admit a gonosomal basis $(f_i)_{i\in \Lambda}\cup h$ such that for all $ i,j \in \Lambda$ we have that $f_if_j=hh=0$ and $f_ih=hf_i=\sum_{k\in \Lambda}\gamma_{ik}f_k+\widetilde{\gamma}_ih$, where $\gamma_{ij},\widetilde{\gamma}_i\geq0$ and $\sum_{k\in \Lambda}\gamma_{ik}+\widetilde{\gamma}_i=1$.

Mainly (but not solely), we will focus on those gonosomal algebras, which can be realised as the commutative duplicate of a baric algebra (see \cite[Subsection 4.2]{V_16}). Next, after revising some necessary background and fixing some notation, we recall what this construction consists in already in the context of genetic systems with a single male genotype.

	A $\mathbb{K}$-algebra $\mathcal{A}$ is called \textit{baric} if it admits a non-zero algebra morphism $\omega\colon\mathcal{A}\longrightarrow\mathbb{K}$, which is called a \textit{weight function} of $\mathcal{A}$.  By \cite[Lemma 1.10]{WB_80}, a $n$-dimensional $\mathbb{K}$-algebra $\mathcal{A}$ is baric if and only if it admits a basis $B=\{e_i\}_{i \in \Lambda}$ such that $e_ie_j=\sum_{k\in \Lambda} \gamma_{ijk}e_k$ with $\sum_{k\in \Lambda}\gamma_{ijk}=1$ for all $i,j \in \Lambda$.
	Moreover, given a commutative $\mathbb{K}$-algebra $\mathcal{A}$ (non-necessarily baric), the quotient space  $D(\mathcal{A})=(\mathcal{A} \otimes \mathcal{A}) / I$, where $I=\spa\{x \otimes y-y \otimes x\colon x,y \in \mathcal{A}\}$  endowed with the component-wise multiplication is said to be the \textit{commutative duplicate} of $\mathcal{A}$.  By $x \otimes y$, we will denote the elements of $D(\mathcal{A})$. Lastly, we recall the surjective morphism $\mu\colon D(\mathcal{A})\longrightarrow\mathcal{A}^2$, $x\otimes y\mapsto xy$, which is called the \textit{Etherington morphism}.
	
	\begin{construction}[{\cite[Proposition 12]{V_16}}] \label{const_1}
		Let $\mathcal{A}$ be a baric $\mathbb{R}$-algebra with basis $\{e_i\}_{i\in \Lambda}$ and multiplication given by $e_je_i=e_ie_j=\sum_{k\in \Lambda}\gamma_{ijk}e_k$ such that $\sum_{k\in \Lambda}\gamma_{ijk}=1$ for any $i,j \in \Lambda$, and $D(\mathcal{A})$ the commutative duplicate of $\mathcal{A}$.
		Let $\Gamma=\Lambda\times\Lambda$ and consider a subset $\Omega\varsubsetneq\Gamma$ and a pair $(k,l)\in\Gamma\backslash\Omega$.
		Let $F=\spa\{f_{rs}=e_r\otimes e_s\colon (r,s)\in\Omega\}$ and $M=\spa\{h=e_k\otimes e_l\}$ be two subspaces of $D(\mathcal{A})$ such that
		\[\mu(F)\otimes\mu(M)=\spa\{e_re_s \otimes e_ke_l\colon(r,s)\in\Omega\}\subset F\oplus M.\] 
		Then, the subspace $F\oplus M\subset D(\mathcal{A})$ with the product given by
		\begin{align*}
			f_{rs}h&:=(e_re_s)\otimes(e_ke_l)=\left(\sum_{p=1}^{n}\gamma_{rsp}e_p\right)\otimes\left(\sum_{q=1}^{n}\gamma_{klq}e_q\right)
			=\sum_{(p,q)\in\Lambda}\gamma_{rsp}\gamma_{klq}f_{pq}+\gamma_{rsk}\gamma_{kll}h,
		\end{align*}
		for any $(r,s)\in\Omega$ and zero in another case, is a gonosomal algebra with basis $(f_{rs})_{(r,s)\in\Lambda}\cup h$. 
	\end{construction}
	
	Nevertheless, the following example shows that not every genetic system with a single male genotype can be modelled using only Construction \ref{const_1}.
	
	\begin{example}[{\cite[Example 16]{V_16}}]\label{ej:art_lem} 
		If we identify $e_1\leftrightarrow X$, $e_2\leftrightarrow X^*$ and $e_3\leftrightarrow Y$, in the particular case of \textit{Dicrostonyx torquatus} (Arctic lemming), the result of one of the crosses is
		\begin{align}\label{prod_artlem}
			(e_2\otimes e_3)(e_1\otimes e_3)=\frac{1}{3}e_1\otimes e_2+\frac{1}{3}e_2\otimes e_3+\frac{1}{3}e_1\otimes e_3.
		\end{align}
		Now, suppose there exists a baric algebra with basis $\{e_1,e_2,e_3\}$ such that gives rise to \eqref{prod_artlem} by following Construction \ref{const_1} where $F=\spa\{e_1\otimes e_1, e_1\otimes e_2, e_2\otimes e_3\}$ and $M=\spa\{e_1\otimes e_3\}$. We can write $\mu(e_2\otimes e_3)=e_2e_3=\alpha_1e_1+\alpha_2e_2+\alpha_3e_3$ and $\mu(e_1\otimes e_3)=e_1e_3=\beta_1e_1+\beta_2e_2+\beta_3e_3$ with $\alpha_1+\alpha_2+\alpha_3=\beta_1+\beta_2+\beta_3=1$. Thus, we have that 
		\begin{align*}
			(e_2\otimes e_3)(e_1\otimes e_3)&=\alpha_1\beta_1e_1\otimes e_1+\alpha_2\beta_2e_2\otimes e_2+\alpha_3\beta_3e_3\otimes e_3\\
			 & \ \, \ {}+(\alpha_1\beta_2+\alpha_2\beta_1)e_1\otimes e_2+(\alpha_1\beta_3+\alpha_3\beta_1)e_1\otimes e_3+(\alpha_2\beta_3+\alpha_3\beta_2)e_2\otimes e_3.
		\end{align*}
		From \eqref{prod_artlem} and the previous expression, we get that $\alpha_1\beta_1=\alpha_2\beta_2=\alpha_3\beta_3=0$ and $\alpha_1\beta_2+\alpha_2\beta_1=\alpha_1\beta_3+\alpha_3\beta_1=\alpha_2\beta_3+\alpha_3\beta_2=\frac{1}{3}$. However, this system of equations does not admit any solution, which yields that such baric algebra does not exist.
	\end{example}
	
	Given the previous example, we revise how, given a gonosomal algebra, we can obtain others by reducing its gonosomal basis (see \cite[Subsection 4.1]{V_16}), already in the framework of genetic systems with a single male genotype, too.

	\begin{construction}[{\cite[Propostion 10]{V_16}}]\label{const_2}
		Let $\mathcal{A}$ be a gonosomal $\mathbb{K}$-algebra with gonosomal basis $(f_i)_{i \in \Lambda}\cup h$ and product given by $f_ih=\sum_{k \in \Lambda}\gamma_{ik}f_k+\widetilde{\gamma}_ih$. If there is a subset $I\subsetneq\Lambda$ such that for all $i\in\Lambda\backslash I$  we have $\sigma_{i}:=1-\sum_{k\in I}\gamma_{ik}\neq0$, then the subspace spanned by $(f_i)_{i\in\Lambda\backslash I}\cup h$ with the product given by 
		\[f_i\ast h=\sigma_{i}^{-1}\left(\sum_{k\in\Lambda\backslash I}\gamma_{ik}f_k+\widetilde{\gamma_i}h\right)\]
		and $f_i\ast f_j=h\ast h=0$ for all $i,j\in\Lambda\backslash I$, is a gonosomal algebra.
	\end{construction}
	In Section \ref{sec:artic_lemming}, we will see how Example \ref{ej:art_lem} can be modelled by combining Constructions~\ref{const_1} and \ref{const_2}. In fact, all the genetic examples considered in this work can be obtained by using any of such procedures or a combination of both.
	
	Next, we recall how a gonosomal algebra is associated with a gonosomal evolution operator and vice versa. Following \cite[Section 10.1]{R_19}, given the inheritance real coefficients $\{\gamma_{ij}\}_{i,j\in\Lambda}$ and $\{\widetilde{\gamma}_i\}_{i\in\Lambda}$ of a gonosomal algebra, the corresponding \textit{gonosomal evolution operator} is defined in coordinate form by $V\colon\mathbb{R}^{n+1}\rightarrow\mathbb{R}^{n+1},(x_1,\dots,x_n,u)\mapsto(x_1^\prime,\dots,x_n^\prime,u^\prime)$,
	\begin{align}\label{gon_op}
		V: \left\{
		\begin{array}{lcl}
			x_k^\prime & = & u\sum_{i \in \Lambda }\gamma_{ik}x_i, \quad
			k \in \Lambda; \medskip\\
			u^\prime & = & u\sum_{i \in \Lambda }\widetilde{\gamma}_ix_i.
		\end{array}
		\right. 
	\end{align}
	As explained in \cite[Section 10.3]{R_19}, the operator $V$ does not map the simplex $S^n$ of $\mathbb{R}^{n+1}$ to itself. For this reason, is necessary to define the corresponding \textit{normalised gonosomal evolution operator} $\widetilde{V}$ with coefficients $\gamma_{ij}$ and $\widetilde{\gamma}_i$ over the set $S^{n}$ in coordinate form as
	\begin{equation}\label{gon_op_norm}
		\widetilde{V}: \left\{
		\begin{array}{l}
			x'_j=\dfrac{u\sum_{i \in \Lambda}\gamma_{ij}x_i}{u\sum_{i \in \Lambda}x_i}, \quad
			j \in \Lambda\medskip\\
			u'=\dfrac{u\sum_{i \in \Lambda}\widetilde{\gamma}_ix_i}{u\sum_{i \in \Lambda}x_i}.
		\end{array}
		\right. 
	\end{equation}
	Notice that this operator $\widetilde{V}$  maps the simplex $S^n$ to itself. 
    Furthermore, for applications in genetics, we work with the subset \[S^{n,1}=\left\{(x_1,\dots,x_n,u)\in\mathbb{R}_{\geq0}^{n+1}\colon \sum_{i \in \Lambda }x_i>0,u>0,\sum_{i \in \Lambda }x_i+u=1\right\}\subset S^n\]
	whose elements can be interpreted as the frequency distributions of the genotypes $f_i$ and $h$. We know that $\widetilde{V}$ also maps $S^{n,1}$ to itself if and only if $(\gamma_{i1},\dots,\gamma_{in},\widetilde{\gamma}_i)\in S^{n,1}$ for any $i \in \Lambda$ (see \cite[Proposition~10.3]{R_19}). Moreover, we have that if $\widetilde{V}$ maps $T\subset S^{n,1}$ to itself then, given an arbitrary initial point $s=(x_1,\dots,x_n,u) \in T $, it holds that $u^{(k)}>0$ for any $k\in \N$. Consequently, the study of the trajectories of the normalised gonosomal evolution operator $\widetilde{V}$ is equivalent to the study of those of the following operator:
   \begin{equation}\label{op_simp}
		 \left\{
		\begin{array}{l}
			x'_j=\dfrac{\sum_{i \in \Lambda}\gamma_{ij}x_i}{\sum_{i \in \Lambda}x_i}, \quad
			j \in \Lambda\medskip\\
			u'=\dfrac{\sum_{i \in \Lambda}\widetilde{\gamma}_ix_i}{\sum_{i \in \Lambda}x_i}.
		\end{array}
		\right. 
	\end{equation}

   From now on, if $T$ is an invariant subset of $S^{n,1}$ with respect to $\widetilde{V}$, we will use both operators \eqref{gon_op_norm} and \eqref{op_simp} interchangeably.

    On the other hand, as shown in \cite[Proposition~10.4]{R_19}, there exists a one-to-one correspondence between non-zero, non-negative fixed points, that is, non-zero fixed points with all non-negative coordinates (equivalently non-negative and normalisable fixed points) of \eqref{gon_op} and the fixed points of \eqref{gon_op_norm}. Concretely, $(x_1,\dots,x_n,u)$ is a non-negative and normalisable fixed point of \eqref{gon_op} if and only if ${(\sum_{i \in \Lambda }x_i+u)}^{-1} (x_1,\dots,x_n,u)$  is a fixed point of \eqref{gon_op_norm}.

    In this work, we only use the operator \eqref{gon_op} to compute the fixed points of  \eqref{gon_op_norm}, which will be essential for studying the limit points of the trajectories
	\[
	\big\{s^{(k)}=\widetilde{V}^k(s)=(x_1^{(k)},\dots,x_n^{(k)},u^{(k)})\big\}_{k\in \N}
	\] 
	for an arbitrary initial point $s=s^{(0)}\in S^{n,1}$ and for drawing the corresponding biological interpretations.

		\section{ZW system with male feminization}\label{wolbachia}
	This section is dedicated to studying the ZW sex-determination system that Woodlice follows (see \cite{CMFRB_04}). \textit{Wolbachia}  is an intracellular maternally inherited bacteria affecting a wide range of arthropods. In the case of woodlice of the \textit{Armadillium vulgare} species, \textit{Wolbachia} is responsible for the feminisation of genetic males. When Wolbachia infects a male with genotype ZZ, denoted by ZZ+w, it becomes a female, which can cross with a male ZZ.  So, in this population, there are three female genotypes: $f_1\leftrightarrow$ ZZ + w, $f_2\leftrightarrow$ ZW and $f_3\leftrightarrow$ ZW+w, and a male genotype $h\leftrightarrow$ ZZ. As explained in detail in \cite[Example 15]{V_16}, the results of crosses can be retrieved by a gonosomal algebra obtained by Construction \ref{const_1} with gonosomal basis $\{f_1,f_2,f_3,h\}$ and product given by 
	\begin{align}\label{al_wolbachia}
		f_1h=\eta f_1+(1-\eta)h,  \quad f_2h=\frac{1}{2}f_2+\frac{1}{2}h,\quad
		f_3h=\frac{\eta}{2}f_1+\frac{1-\eta}{2}f_2+\frac{\eta}{2}f_3+\frac{1-\eta}{2}h
	\end{align}
	where $\eta$ ($\frac{1}{2} < \eta < 1$) denotes the transmission rate of Wolbachia in the offspring. Consider the associated gonosomal evolution operator $V_\eta$, with $\frac{1}{2} \leq \eta \leq 1$, given by
	\begin{equation}\label{op_wolbachia}
		V_{\eta}: \left\{
		\begin{array}{rclrcl}
			x_1'&=&\eta x_1u+\frac{\eta}{2}x_3u; & x_3'&=& \frac{\eta}{2}x_3u;\medskip \\
			x_2'&=&\frac{1}{2}x_2u+\frac{1-\eta}{2}x_3u; & u'&=&(1-\eta) x_1u+\frac{1}{2}x_2u+\frac{1-\eta}{2}x_3u.
		\end{array}
		\right.
	\end{equation}
	
	\begin{proposition}\label{prop:fix_points_wolbachia}
    The operator $V_\eta$, with $\frac{1}{2} < \eta < 1$, has three non-zero fixed points: $(\frac{2}{2-\eta},\frac{2}{2-\eta},\frac{-2}{2-\eta},\frac{2}{\eta})$, $(0,2,0,2)$ and $(\frac{1}{1-\eta},0,0,\frac{1}{\eta})$.
    
		Moreover, $V_{\frac{1}{2}}$ and  $V_{1}$ have infinitely many non-zero fixed points: the first one has $(\frac{4}{3},\frac{4}{3},-\frac{4}{3},4)$ and the family $(\rho,2-\rho,0,2)$ with $\rho\in\mathbb{R}$ and the second one has the family $(\beta,2,-\beta,2)$ with  $\beta\in\mathbb{R}$.
	\end{proposition}
	\begin{proof}
		We need to solve the system of equations given by
		\begin{equation}\label{sist_fix_points_wolb}
			\left\{
			\begin{array}{rclrcl}
				x_1&=&\eta x_1u+\frac{\eta}{2}x_3u;\quad  & x_3&=& \frac{\eta}{2}x_3u; \medskip\\
				x_2&=&\frac{1}{2}x_2u+\frac{1-\eta}{2}x_3u;\quad &
				u&=&(1-\eta) x_1u+\frac{1}{2}x_2u+\frac{1-\eta}{2}x_3u.
			\end{array}
			\right.
		\end{equation}
		First, assume that $\frac{1}{2}<\eta<1$. If $x_3=0$, it is deduced that $x_1=0$ or $u=\frac{1}{\eta}$. On the one hand, if $x_1=0$, we get either $u=0$ or $x_2=2$. But if $x_3=x_1=u=0$, then necessarily $x_2=0$. However, if $x_2=2$, we have the non-trivial solution $(0,2,0,2)$. On the other hand, if $u=\frac{1}{\eta}$ then we obtain that $x_2=0$ because $\eta\neq\frac{1}{2}$. Finally, as $\eta\neq1$, then $x_1=\frac{1}{1-\eta}$. Therefore, $(\frac{1}{1-\eta},0,0,\frac{1}{\eta})$ is another solution of \eqref{sist_fix_points_wolb}. Next, assume that $x_3\neq0$. Hence, $u=\frac{2}{\eta}$. This implies that $x_1=x_2=-x_3$. We conclude that $x_1=\frac{2}{2-\eta}$. Thus, we obtain the solution $(\frac{2}{2-\eta},\frac{2}{2-\eta},\frac{-2}{2-\eta},\frac{2}{\eta})$.

        Now, let's the extreme cases $\eta=\frac{1}{2}$ and $\eta=1$. It is easy to check that, for $\eta=\frac{1}{2}$, the solutions of \eqref{sist_fix_points_wolb} are the points $(\frac{4}{3},\frac{4}{3},-\frac{4}{3},4)$ and $(\rho,2-\rho,0,2)$ with $\rho\in\mathbb{R}$ and for $\eta=1$ the solutions are the family of points  $(\beta,2,-\beta,2)$ with  $\beta\in\mathbb{R}$.
			\end{proof}

		Once the fixed points have been obtained, we consider the normalised version of \eqref{op_wolbachia}, that is 
	\begin{equation} \label{op_wolbachia_norm}
		\widetilde{V}_\eta:\left\{
		\begin{array}{rclcrcl}
			x_1'&=&\frac{\eta u( 2 x_1+ x_3)}{2u(x_1+x_2+x_3)}; &\quad& x_3'&=&  \frac{\eta ux_3}{2u(x_1+x_2+x_3)};\medskip\\
			x_2'&=&\frac{u(x_2+(1-\eta)x_3)}{2u(x_1+x_2+x_3)}; &\quad& u'&=&  \frac{u((1-\eta)(2x_1+x_3)+x_2)}{2u(x_1+x_2+x_3)}.\\	
		\end{array}
		\right.
	\end{equation}
    Next, we study the limit points of the different trajectories for arbitrary initial points.

    \vspace{0.3cm}
   \textbf{Case 1: $\frac{1}{2}<\eta<1$.}
    
    In this case, it is clear that $\widetilde{V}_\eta$ maps $S^{3,1}$ to itself and that its fixed points are $(0,\frac{1}{2},0,\frac{1}{2})$ and $(\eta,0,0,1-\eta)$.

	\begin{proposition}\label{lem:desig}
		Consider the operator $\widetilde{V}_\eta$ defined by \eqref{op_wolbachia_norm} with $\frac{1}{2}<\eta<1$. Then, for any initial point $s=(x_1^{(0)},x_2^{(0)},x_3^{(0)},u^{(0)}) \in S^{3,1}$, the following assertions hold:
		\begin{enumerate}[\rm (i)]
			\item $ \frac{1}{2} \leq x_1^{(k)}+x_2^{(k)} \leq \eta$ for all $k\geq 1$;\smallskip
			\item $x_3^{(k)} \leq \eta$ and $x_3^{(k+1)}\leq x_3^{(k)}$  for all $k\geq 1$;\smallskip
			\item there exists a natural number $n_0$ such that $x_2^{(k)}+(1+\eta)x_3^{(k)}  \leq 2 \eta (x_2^{(k)} + x_3^{(k)})$ for all $k \geq n_0$; and\smallskip
			\item there exists a natural number $m_0$ such that $x_1^{(k+1)}\geq x_1^{(k)}$  for all $k\geq m_0$.
		\end{enumerate} 
	\end{proposition}
	
	\begin{proof}
		Take $s=(x_1^{(0)},x_2^{(0)},x_3^{(0)},u^{(0)}) \in S^{3,1}$. First, we introduce the notation $\alpha^{(k)}:=x_1^{(k)}+x_2^{(k)}+x_3^{(k)}$ for any $k\geq0$.
		For item (i) and item (ii), consider an integer number $k\geq1$. Then, we observe that 
		\[\frac{1}{2}\leq\frac{\alpha^{(k-1)}+(2\eta-1)x_1^{(k-1)}}{2\alpha^{(k-1)}}\leq\frac{\alpha^{(k-1)}+(2\eta-1)\alpha^{(k-1)}}{2\alpha^{(k-1)}}\leq\eta,\]
		where the second term is actually $x_1^{(k)}+x_2^{(k)}$, what yields the claim.
		
		For item (ii), observe that $2\alpha^{(k)}\geq 2(x_1^{(k)}+x_2^{(k)})\geq1$ using item (i). Then, 
		\[x_1^{(k)}+x_3^{(k)}=\frac{\eta(x_1^{(k-1)}+x_3^{(k-1)})}{\alpha^{(k-1)}}\leq\eta\quad\text{and}\quad x_3^{(k+1)}=\frac{\eta x_3^{(k)}}{2\alpha^{(k)}} \leq \eta x_3^{(k)}.\]
		From the first one, we get that $x_3^{(k)} \leq x_1^{(k)}+x_3^{(k)}\leq\eta$. From the second one, as $\eta x_3^{(k)}<\eta$ and $\eta x_3^{(k)}<x_3^{(k)}$, we are done.

		For item (iii), we first claim that there exists a number $n_0\in\mathbb{N}$ such that $x_2^{(n_0)}+(1+\eta)x_3^{(n_0)}  \leq 2 \eta (x_2^{(n_0)} + x_3^{(n_0)})$. Indeed, by contrary, assume that $x_2^{(k)}+(1+\eta)x_3^{(k)}  > 2 \eta (x_2^{(k)} + x_3^{(k)})$ for all $k \in \mathbb{N}$. Equivalently, we have that  $x_3^{(k)}(1-\eta) > x_2^{(k)}(2 \eta -1)$ for all $k \in \mathbb{N}$. Consequently,
		\[
		(1- \eta ) \frac{\eta x_3^{(k-1)}}{2\alpha^{(k-1)}} > (2\eta -1)\frac{x_2^{(k-1)}+(1- \eta)x_3^{(k-1)}}{2\alpha^{(k-1)}}
		\]
		which is equivalent to $(1- \eta)^2 x_3^{(k-1)} > (2\eta -1) x_2^{(k-1)} $ for all $k \in \mathbb{N}$. In the same way, we get that $(1- \eta)(1-\eta-\eta^2) x_3^{(k-2)} > (2\eta -1) x_2^{(k-2)} $. Inductively, it is easy to check that 
		\begin{align}\label{eq:ineq}
			x_3^{(k-l)} (1- \eta) \left(1-\sum_{j=1}^l \eta ^j \right) > (2 \eta -1) x_2^{(k-l)}
		\end{align}
		for all $l \in \mathbb{N},l<k$. Hence, as $\frac{1}{2}<\eta<1$, the series $\sum_{j=1}^{\infty}\eta^j=\frac{\eta}{1-\eta}>1$. Therefore, there exists a number $m\in\mathbb{N}$ large enough such that $\sum_{j=1}^{m}\eta^j>1$. Then, taking a number $k\in\mathbb{N},k\geq m$, we obtain that $x_3^{(k-m)} (1- \eta) (1-\sum_{j=1}^{m} \eta ^j )<0$, which is a contradiction with the fact that $(2 \eta -1) x_2^{(k-m)}\geq0$ and \eqref{eq:ineq}.
        
		 Moreover, looking at the previous calculations, we have just proved that if $x_3^{(k)}(1-\eta)^2 \leq x_2^{(k)}(2 \eta -1)$ for certain $k \in \N$, then $x_3^{(k+1)}(1-\eta) \leq x_2^{(k+1)}(2 \eta -1)$. 
                But, taking into account that $\frac{1}{2}<\eta<1$ and using that there exists a number $n_0\in\mathbb{N}$ such that $x_3^{(n_0)}(1-\eta) \leq x_2^{(n_0)}(2 \eta -1)$, we get that $x_3^{(n_0)}(1-\eta)^2 < x_3^{(n_0)}(1-\eta) \leq x_2^{(n_0)}(2 \eta -1)$, which completes the proof.
		
		Lastly, for proving item (iv), we know that there exists a $n_0\in \N$ large enough such that $x_2^{(k)}+(1+\eta)x_3^{(k)}  \leq 2 \eta (x_2^{(k)} + x_3^{(k)})$ for all $k \geq n_0$ by item (iii). Now, we have that $\alpha^{(k)}\leq \eta $ for all $k \geq n_0+1$. Indeed,   
		\[
		\alpha^{(k)}= \frac{2 \eta x_1^{(k-1)} + (1+ \eta) x_3^{(k-1)} + x_2^{(k-1)}}{2\alpha^{(k-1)}} \leq \frac{2\eta(x_1^{(k-1)}+x_2^{(k-1)}+x_3^{(k-1)})}{2\alpha^{(k-1)}}= \eta
		\]
		for all $k \geq n_0+1$. Then, taking $m_0=n_0+1$ we obtain that $2x_1^{(k)} (\alpha^{(k)} - \eta) - \eta x_3^{(k)} \leq 0$ for all $k \geq m_0$, which implies that $2\alpha^{(k)}x_1^{(k)} \leq \eta(2 x_1^{(k)} + x_3^{(k)})$ for all $k \geq m_0$. Therefore, 
		\[
		x_1^{(k+1)}= \frac{\eta(2 x_1^{(k)} + x_3^{(k)})}{2\alpha^{(k)}} \geq x_1^{(k)} 
		\]
		for all $k \geq m_0$, and the result follows. 
	\end{proof}
        \begin{remark}\label{rem:limit_u}
           Let be the operator $\widetilde{V}$  defined in \eqref{op_wolbachia_norm} and suppose that $\lim_{k\to\infty}x_1^{(k)}$,  
            $\lim_{k\to\infty}x_2^{(k)}$ and $\lim_{k\to\infty}x_3^{(k)}$ exist for any initial point $s=(x_1^{(0)},x_2^{(0)},x_3^{(0)},u^{(0)}) \in S^{3,1}$. Then, observe that $\lim_{k\to\infty}u^{(k)} \neq 0$. Indeed, if suppose that $\lim_{k\to\infty}u^{(k)}=0$ then $\lim_{k\to\infty}x_1^{(k)}+x_2^{(k)}+x_3^{(k)}=1$ and necessarily $\lim_{k\to\infty}(1-\eta)(2x_1^{(k)}+x_3^{(k)})+x_2^{(k)}=0$. As $\eta\neq1$, it holds that $\lim_{k\to\infty}x_1^{(k)}=\lim_{k\to\infty}x_2^{(k)}=\lim_{k\to\infty}x_3^{(k)}=0$, a contradiction with the fact that $x_1^{(k)}+x_2^{(k)}+x_3^{(k)}+u^{(k)}=1$ for any $k\geq0$.
        \end{remark}
	\begin{theorem}
		Consider the operator $\widetilde{V}_\eta$ given by \eqref{op_wolbachia_norm} with $\frac{1}{2}<\eta<1$. Then, for an initial point $s=(x_1^{(0)},x_2^{(0)},x_3^{(0)},u^{(0)}) \in S^{3,1}$, it holds that
		\[
		\lim_{k \to \infty} \widetilde{V}_\eta^{k}(s)= \left\{
		\begin{array}{ll}
			(0,\frac{1}{2},0,\frac{1}{2}), & {\rm if} \ x_1^{(0)}=x_3^{(0)}=0, \smallskip \\
			(\eta,0,0, 1- \eta), & {\rm otherwise.}  
		\end{array}
		\right.
		\]
	\end{theorem}
	
	\begin{proof}
		Consider $s=(x_1^{(0)},x_2^{(0)},x_3^{(0)},u^{(0)}) \in S^{3,1}$. First, note that if $x_1^{(0)}=x_3^{(0)}=0$, then clearly $\widetilde{V}_\eta^{k}(s)=(0,\frac{1}{2},0,\frac{1}{2})$ for all $k\geq1$ and so $\lim_{k \to \infty} \widetilde{V}_\eta^{k}(s)=(0,\frac{1}{2},0,\frac{1}{2})$. 
				
		Otherwise, we can ensure $x_1^{(1)}\neq0$. By Lemma \ref{lem:desig} (iv), we have that $\{x_1^{(k)}\}_{k \in \N}$ is an increasing bounded positive sequence. Consequently,  its limit exists and is positive. Analogously, by Lemma \ref{lem:desig} (ii), we have that $\{x_3^{(k)}\}_{k\in \N}$ is a descending bounded positive sequence, so its limit exists as well. Since 
        \[x_2^{(k)}=\frac{\eta(2x_1^{(k)}+x_3^{(k)})}{2x_1^{(k+1)}}-x_1^{(k)}-x_3^{(k)}\quad\text{and}\quad u^{(k)}=1-(x_1^{(k)}+x_2^{(k)}+x_3^{(k)})\]
        for all $k\geq1$, we deduce that  $\{x_2^{(k)}\}_{k\in \N}$ and   $\{u^{(k)}\}_{k\in \N}$ are also two convergent sequences. 
        As a consequence of Remark \ref{rem:limit_u}, if we take limits on both sides of the expressions which define the dynamical system \eqref{op_wolbachia_norm}, we have that the possible limits are exactly its fixed points but, as  $\lim_{k \to \infty} x_1^{(k)}>0$, we conclude that $\lim_{k \to \infty} \widetilde{V}_\eta^{k}(s)=(\eta,0,0,1-\eta)$.
        
       	\end{proof}
	   
   {\textbf{Case 2: $\eta=\frac{1}{2}$.}

   In this case, we also have that $S^{3,1}$ is invariant under $\widetilde{V}_{\frac{1}{2}}$. We obtain the following result regarding its limit points.
	\begin{theorem}
    Consider the operator $\widetilde{V}_{\frac{1}{2}}$ defined by \eqref{op_wolbachia_norm}. Then,  for an initial point $s=(x_1^{(0)},x_2^{(0)},x_3^{(0)},u^{(0)}) \in S^{3,1}$, it holds that
		\[
		\lim_{k \to \infty} \widetilde{V}_\frac{1}{2}^k(s)= \left\{
		\begin{array}{ll}
			\widetilde{V}_\frac{1}{2}(s), & {\rm if} \ x_3^{(0)}=0, \medskip \\
			\left(\frac{x_1^{(1)}+x_3^{(1)}}{1+4x_3^{(1)}},\frac{1+2(x_3^{(1)}-x_1^{(1)})}{2(1+4x_3^{(1)})},0, \frac{1}{2}\right), & {\rm otherwise.}  
		\end{array}
		\right.
		\] 
	\end{theorem} 
		
	\begin{proof}
		Let  $s=(x_1^{(0)},x_2^{(0)},x_3^{(0)},u^{(0)}) \in S^{3,1}$. If $x_3^{(0)}=0$, then it is easy to check that
		\[
		\widetilde{V}_\frac{1}{2}^k(s)=\widetilde{V}_\frac{1}{2}(s)=\left(\frac{x_1^{(0)}}{2(x_1^{(0)}+x_2^{(0)})},\frac{x_2^{(0)}}{2(x_1^{(0)}+x_2^{(0)})},0,\frac{1}{2}\right)
		\]
		for all $k \geq 1$. Consequently, $\lim_{k \to \infty} \widetilde{V}_\frac{1}{2}^{k}(s)=\widetilde{V}_\frac{1}{2}(s)$. 
		
		Next, assume that $x_3^{(0)}\neq0$. Then, as $x_1^{(k)}+x_2^{(k)}=\frac{1}{2}$ for any $k\geq1$, we have that 
		\[
		x_1^{(k)}=\frac{ x_1^{(k-1)}+ \frac{1}{2}x_3^{(k-1)}}{1+2x_3^{(k-1)}}=\frac{ x_1^{(k-2)}+ (\frac{1}{2}+\frac{1}{4})x_3^{(k-2)}}{1+(2+1)x_3^{(k-2)}}=\dots=\frac{ x_1^{(1)}+x_3^{(1)} \sum_{j=1}^{k-1}\frac{1}{2^j}}{1+x_3^{(1)}\sum_{j=-1}^{k-3}\frac{1}{2^j}}
 \]
		for all $k\geq2$. Hence,
		\[
		\lim_{k \to \infty}x_1^{(k)}=\frac{ x_1^{(1)}+x_3^{(1)} \sum_{k=1}^{\infty}\frac{1}{2^k}}{1+x_3^{(1)}\sum_{k=-1}^{\infty}\frac{1}{2^k}}=\frac{x_1^{(1)}+x_3^{(1)}}{1+4x_3^{(1)}}\quad\text{and}\quad\lim_{k \to \infty}x_2^{(k)}=\frac{1+2(x_3^{(1)}-x_1^{(1)})}{2(1+4x_3^{(1)})}.
		\] 
		Analogously, we have that
		\[
		x_3^{(k)}=\frac{\frac{1}{2}x_3^{(k-1)}}{1+2x_3^{(k-1)}}=\frac{\frac{1}{4}x_3^{(k-2)}}{1+(2+1)x_3^{(k-2)}}=\dots=\frac{\frac{1}{2^{k-1}}x_3^{(1)}}{1+x_3^{(1)}\sum_{j=-1}^{k-3}\frac{1}{2^j}}
		\]
		for all $k\geq2$ and, so $\lim_{k\to\infty}x_3^{(k)}=0$. Finally, $\lim_{k\to\infty}u^{(k)}=\frac{1}{2}$, and the result follows.
	\end{proof}
        
        \vspace{0.3cm}
        \textbf{Case 3: $\eta=1$.}
        
        First, notice that, unlike the previous cases, $\widetilde{V}_1$ does not necessarily map $S^{3,1}$ to itself. In fact, we have the following result.
        \begin{lemma}
            The set $\{(x_1,x_2,x_3,u)\in S^{3,1} \colon x_2>0\}$ is the largest invariant subset of $S^{3,1}$ with respect to $\widetilde{V}_1$.
        \end{lemma}
        \begin{proof}
            Clearly, if we take an initial point $s=(x_1^{(0)},x_2^{(0)},x_3^{(0)},u^{(0)})\in S^{3,1}$ such that $x_2^{(0)}>0$, then $x_2^{(k)}=u^{(k)}>0$ for any $k\in\mathbb{N}$.
            To prove that it is the largest invariant contained in $S^{3,1}$, consider an initial point $s\in S^{3,1}$ with $x_2^{(0)}=0$. So, $u^{(1)}=0$  and thus $s^{(1)}\notin S^{3,1}$.
        \end{proof}
                
	\begin{theorem}
	    Consider the operator $\widetilde{V}_1$ defined by \eqref{op_wolbachia_norm}. Then, for an initial point $s=(x_1^{(0)},x_2^{(0)},x_3^{(0)},u^{(0)}) \in S^{3,1}$ with $x_2^{(0)}>0$, it holds that
        \[ 		
        \lim_{k \to \infty} \widetilde{V}_1^k(s)= \left\{ 		\begin{array}{ll} 			(0,\frac{1}{2},0,\frac{1}{2}), & {\rm if} \ x_1^{(0)}=x_3^{(0)}=0, \medskip \\ 			(1,0,0,0), & {\rm otherwise.}   		\end{array} 		\right. 
        \]
	\end{theorem}
	\begin{proof}
        Take an arbitrary initial point $s=(x_1^{(0)},x_2^{(0)},x_3^{(0)},u^{(0)}) \in S^{3,1}$ with $x_2^{(0)}>0$. Firstly, notice that if $x_1^{(0)}=x_3^{(0)}=0$, then clearly $\widetilde{V}_1^k(s)=(0,\frac{1}{2},0,\frac{1}{2})$ for any $k\geq1$, and so $\lim_{k \to \infty}\widetilde{V}_1^k(s)=(0,\frac{1}{2},0,\frac{1}{2})$.
        
        Otherwise, we can ensure that $x_1^{(1)}\neq0$. Moreover,  we have that
        \[
        x_i^{(k)}=\frac{x_i^{(k-1)}}{2(x_1^{(k-1)}+x_2^{(k-1)}+x_3^{(k-1)})}=\dots=\frac{x_i^{(1)}}{2^{k-1}(x_1^{(1)}+x_3^{(1)})+2x_2^{(1)}},
        \]
        for $i=2,3$ and for all $k\geq2$. This implies that $\lim_{k\to\infty}x_2^{(k)}=\lim_{k\to\infty}x_3^{(k)}=0$. In addition, since $x_2^{(k)}=u^{(k)}$ for all $k\geq1$, we also have that $\lim_{k\to\infty}u^{(k)}=0$. Thus, $\lim_{k \to \infty}\widetilde{V}_1^k(s)=(1,0,0,0)$.
        
        	\end{proof}

	\begin{interp}
	    	For  \textit{Armadillium vulgare} population, we can conclude that for any initial state $s\in S^{3,1}$ (the probability distribution on the set of possible genotypes $\{\text{ZZ+w},\text{ZW},\text{ZW+w},\text{ZZ}\}$), the future of the population is always stable. If the transmission rate is $\frac{1}{2}<\eta<1$, then if there is no ZZ+w and ZW+w in the initial state, the population tends to the equilibrium state $(0,\frac{1}{2},0,\frac{1}{2})$, where ZW and ZZ are distributed equally. Otherwise, the population tends to the equilibrium state $(\eta,0,0,1-\eta)$. Notice that, in both cases, the individuals ZW+w tend to extinction. In the case of $\eta=\frac{1}{2}$, although the limit always exists, it depends on the initial point. Moreover, again, the ZW+w individuals always tend to extinction.
        Finally,  for $\eta=1$ and if the probability of genotypes ZZ+w and ZW+w is zero, then the population tends to the equilibrium state $(0,\frac{1}{2},0,\frac{1}{2})$ where ZW and ZZ are distributed equally. Otherwise, the population tends to $(1,0,0,0)$, meaning that the proportion of males gradually decreases and tends to zero. As a result, the population will eventually face extinction due to the insignificant proportion of males.
        	\end{interp}

		\section{XY systems with fertile XY females}\label{lemming}
	
		\subsection{Generalising the wood lemming sex-determination system}

	This subsection aims to develop the dynamic aspects of a (normalised) gonosomal operator that generalises the modelling of some populations with atypical fertile XY females, such as some rodents. Consider a gonosomal algebra $\mathcal{A}$ with gonosomal basis $(f_i)_{i \in \Lambda}\cup h$ and scalars $0\leq\gamma_i\leq\frac{1}{2}$, some non-zero, such that the product is given by  
	\begin{align}\label{al_gen_musminutoides}
		f_ih=\gamma_{i}f_1+\sum_{j=2}^{n}\frac{1-2\gamma_i}{n-1}f_j+\gamma_ih,
	\end{align}
	for any $i\in\Lambda$. Its associated gonosomal operator is
	\begin{equation}\label{gen_musminutoides}
		V_{\gamma_1,\dots,\gamma_n}\colon\left\{
		\begin{array}{lcccl}
			x_1^\prime & = &u^\prime & = & u\sum_{i \in \Lambda}\gamma_i x_i,\\
			x_j' &   & & = & u\sum_{i \in \Lambda}\frac{1-2\gamma_i}{n-1}x_i, \text{ for any }j=2,\dots,n;
		\end{array}\right.
	\end{equation}
	We first have to compute the fixed points of the gonosomal operator given by \eqref{gen_musminutoides} depending on the parameters $\{\gamma_i\}_{i \in \Lambda}$. For that, as $x_1^{(k)}=u^{(k)}$ and $x_2^{(k)}=\dots=x_n^{(k)}$ for any $k\geq1$, we can equivalently compute the fixed points on the simplified two-dimensional discrete-time dynamical system given by 
	\begin{equation}\label{gen_musminutoides_simp}
		SV_{n,\gamma,\lambda}: \left\{
		\begin{array}{R{1.3em}l}
			x'=&x\left(\gamma x+\lambda y\right),\\
			y'=&x\left(\frac{1-2\gamma}{n-1}x+\big(1-\frac{2\lambda}{n-1}\big)y\right);
		\end{array}\right.
	\end{equation}
	where $\gamma=\gamma_1$, $\lambda=\sum_{i=2}^n\gamma_i$ and $(\gamma,\lambda)\neq(0,0)$.
	Clearly, there is a one-to-one correspondence between the fixed points of \eqref{gen_musminutoides_simp} and the fixed points of \eqref{gen_musminutoides}. Moreover, this correspondence also holds if we restrict to non-zero, non-negative fixed points.
   	\begin{proposition}\label{prop:fix_points_musminutoides}
		Consider the discrete-time dynamical system $SV_{n,\gamma,\lambda}$ defined by \eqref{gen_musminutoides_simp}, where $n\in\mathbb{N},n>1$, $0\leq\gamma\leq\frac{1}{2}$, $0\leq\lambda\leq\frac{n-1}{2}$ and $(\gamma,\lambda)\neq(0,0)$. Then, $SV_{n,\gamma,\lambda}$ has the following non-zero fixed points:
		\begin{enumerate}[\rm (i)] 
			\item if $\lambda=0$ and $\gamma\neq0$, then the only  fixed point is $\left(\frac{1}{\gamma},\frac{1-2\gamma}{\gamma(\gamma-1)(n-1)}\right)$; and\smallskip
			\item if $\lambda\neq0$, then we distinguish two  cases:\smallskip
			\begin{enumerate}[\rm (a)]           
				\item  if $\lambda=(n-1)\gamma$, then the only  fixed point is $\left(\frac{1}{1-\gamma},\frac{1-2\gamma}{\gamma(1-\gamma)(n-1)}\right)$; and\smallskip
				\item \label{apartadob}if $\lambda\neq(n-1)\gamma$, then there exist at most two  fixed points. In fact, $\left(x,\frac{1-\gamma x}{\lambda}\right)$ is a fixed point if and only if $x$ is a solution of the quadratic equation 
				\begin{align}\label{eq_grad2}
					\left(\lambda-(n-1)\gamma\right)x^2+\left((n-1)(\gamma+1)-2\lambda\right)x-n+1=0.
				\end{align}			
			\end{enumerate}
		\end{enumerate}
	\end{proposition}
	\begin{proof}
		We need to solve the system of equations defined by
		\begin{equation*}
			\left\{
			\begin{array}{rcl}
				x&=&x\left(\gamma x+\lambda y\right),\smallskip\\
				y&=&x\left(\frac{1-2\gamma}{n-1}x+\big(1-\frac{2\lambda}{n-1}\big)y\right).
			\end{array}\right.
		\end{equation*}
		First, if $x=0$, then $y=0$. For $x\neq 0$, it holds that $\gamma x+\lambda y=1$. If $\lambda=0$, then $\gamma\neq0$ and $x=\frac{1}{\gamma}$. So, from the second equation, we get that $y=\frac{1-2\gamma}{\gamma(\gamma-1)(n-1)}$. 
		If $\lambda\neq0$, then $y=\frac{1-\gamma x}{\lambda}$ and replacing this expression in the second equation, we obtain precisely the equation~\eqref{eq_grad2}.
		Now, we have to consider two cases. For $\lambda=(n-1)\gamma$ we have that $x=\frac{1}{1-\gamma}$ and $y=\frac{1-2\gamma}{\gamma(1-\gamma)(n-1)}$. 
		In the case of $\lambda\neq(n-1)\gamma$, observe that the discriminant of the equation \eqref{eq_grad2} is
		$\left((n-1)\gamma-2\lambda\right)^2+(n-1)^2(1-2\gamma)\geq0$, and so,
		both solutions are also real, which yields the claim.
	\end{proof}
	Next, we study sufficient and necessary conditions for the non-zero fixed points to be non-negative.
   
	We will denote by $\Delta:=\left((n-1)\gamma-2\lambda\right)^2+(n-1)^2(1-2\gamma)$ the radicand of \eqref{eq_grad2}.
	\begin{proposition}
		Consider the discrete-time dynamical system $SV_{n,\gamma,\lambda}$ defined by \eqref{gen_musminutoides_simp}, where $n\in\mathbb{N},n>1$, $0\leq\gamma\leq\frac{1}{2}$, $0\leq\lambda\leq\frac{n-1}{2}$ and $(\gamma,\lambda)\neq(0,0)$. The non-zero fixed points of $SV_{n,\gamma,\lambda}$ satisfy the following assertions:
		\begin{enumerate}[\rm (i)]
			\item  if $\lambda=0$, then the only fixed point is non-negative if and only if $\gamma=\frac{1}{2}$;
			\item  if $\lambda\neq0$ and $\lambda=(n-1)\gamma$, then the only  fixed point is always non-negative;
			\item if $\lambda\neq0$, then we state the following:
			\begin{enumerate}[\rm (a)]
				\item   if $\lambda>(n-1)\gamma$, then we only have one non-negative fixed point; and
				\item if $\lambda<(n-1)\gamma$, then  every fixed point is non-negative if and only if $\gamma=\frac{1}{2}$. Otherwise, at least one non-negative fixed point exists.
			\end{enumerate}
		\end{enumerate}
	\end{proposition}
	\begin{proof}
		Items (i) and (ii) are consequences of Proposition \ref{prop:fix_points_musminutoides}. To prove item (iii), denote by 	\[x_1=\frac{2\lambda-(n-1)(\gamma+1)-\sqrt{\Delta}}{2\big(\lambda-(n-1)\gamma\big)}\quad\text{and}\quad x_2=\frac{2\lambda-(n-1)(\gamma+1)+\sqrt{\Delta}}{2\big(\lambda-(n-1)\gamma\big)}\] both solutions of \eqref{eq_grad2}. Assume that $\lambda>(n-1)\gamma$. Since $2\lambda \leq n-1$ and $(n-1)(1+\gamma)+\sqrt{\Delta}>n-1$ then $x_1<0$. Hence,   $(x_1,\frac{1-\gamma x_1}{\lambda})$ is not a non-negative fixed point. Now, we prove that  $x_2>0$. We have that:
		\begin{align}\label{eq_x2_pos}
			\begin{split}
				x_2>0&\Longleftrightarrow\sqrt{\Delta}>(n-1)(\gamma+1)-2\lambda\\
				&\Longleftrightarrow (n-1)^2\gamma^2+4\lambda^2-4\lambda\gamma(n-1)+(n-1)^2(1-2\gamma)\\
				&\phantom{\Longleftrightarrow}>(n-1)^2(\gamma+1)^2+4\lambda^2-4\lambda(\gamma+1)(n-1)\\
				&\Longleftrightarrow-4(n-1)^2\gamma>-4(n-1)\lambda\Longleftrightarrow(n-1)\gamma<\lambda,
			\end{split}
		\end{align}
		but, this is  true by hypothesis. Moreover, if $\gamma\neq0$,  then $x_2\leq \frac{1}{\gamma}$.  Indeed,
		\begin{align}\label{eq_x2_1/gamma}
			\begin{split}
				x_2\leq\frac{1}{\gamma}&\Longleftrightarrow \sqrt{\Delta}\leq\frac{1-\gamma}{\gamma}\big(2\lambda-\gamma(n-1)\big)\\
				&\Longleftrightarrow (n-1)^2\gamma^2+4\lambda^2-4\lambda\gamma(n-1)+(n-1)^2(1-2\gamma)\\
				&\phantom{\Longleftrightarrow}\leq\frac{(1-\gamma)^2}{\gamma^2}\big(4\lambda^2+\gamma^2(n-1)^2-4\lambda\gamma(n-1)\big)\\
				&\Longleftrightarrow0\leq\frac{4\lambda}{\gamma^2}\big(1-2\gamma\big)\big(\lambda-\gamma(n-1)\big)
			\end{split}
		\end{align}
		and this last inequality is  true by hypothesis. Therefore, only the fixed point $(x_2,\frac{1-\gamma x_2}{\lambda})$ is non-negative.
		
		Next, suppose that $\lambda<(n-1)\gamma$. It is obvious that $x_2\leq x_1$. Reasoning similarly as we did in \eqref{eq_x2_pos}, we get that $x_2>0$ (and, consequently, $x_1> 0$). Moreover,  $x_1\leq\frac{1}{\gamma}$ if and only if $\gamma=\frac{1}{2}$. Indeed, analogously to \eqref{eq_x2_1/gamma}, we have:
		\begin{align}\label{eq_x1_1/gamma}
			\begin{split}
				x_1\leq\frac{1}{\gamma}&\Longleftrightarrow \sqrt{\Delta}\leq\frac{1-\gamma}{\gamma}\big(\gamma(n-1)-2\lambda\big)\Longleftrightarrow0\leq\frac{4\lambda}{\gamma^2}\big(1-2\gamma\big)\big(\lambda-\gamma(n-1)\big)
			\end{split}
		\end{align}
		and this is verified if and only if $\gamma=\frac{1}{2}$. Therefore,   $(x_1,\frac{1-\gamma x_1}{\lambda})$ and $(x_2,\frac{1-\gamma x_2}{\lambda})$ are both non-negative fixed points for $\gamma=\frac{1}{2}$. In any case,  as $\lambda<(n-1)\gamma$, observe that 
		\[x_2\leq\frac{1}{\gamma}\Longleftrightarrow \sqrt{\Delta}\geq\frac{1-\gamma}{\gamma}\big(\gamma(n-1)-2\lambda\big).\]
		Note that if $\frac{1-\gamma}{\gamma}\big(\gamma(n-1)-2\lambda\big)\leq 0$, the inequality is true. For $\frac{1-\gamma}{\gamma}\big(\gamma(n-1)-2\lambda\big)> 0$, following a reasoning similar to \eqref{eq_x1_1/gamma}, we obtain that  $$\sqrt{\Delta}\geq\frac{1-\gamma}{\gamma}\big(\gamma(n-1)-2\lambda\big)\Longleftrightarrow\frac{4\lambda}{\gamma^2}\big(1-2\gamma\big)\big(\lambda-\gamma(n-1)\big)\leq0,$$ but this is always true. Hence,  the fixed point $(x_2,\frac{1-\gamma x_2}{\lambda})$ is non-negative.
	\end{proof}
	\begin{corollary}\label{cor:GO_norm}
		Consider the gonosomal operator $V_{\gamma_1,\dots,\gamma_n}$ defined by \eqref{gen_musminutoides}, where $n\in\mathbb{N}$, $n>1$, $0\leq\gamma_i\leq\frac{1}{2}$ for any $i\in\Lambda$ and some non-zero. Then, every non-zero fixed point is non-negative  if and only if $\gamma_1=\frac{1}{2}$ or $\sum_{i=2}^n\gamma_i=(n-1)\gamma_1$.
	\end{corollary}
	
	Next, we consider the normalised version of the gonosomal operator defined by \eqref{gen_musminutoides}, that is, 
	\begin{equation}\label{norm_musminutoides}
		\widetilde{V}_{\gamma_1,\dots,\gamma_n}:\left\{
		\begin{array}{lcccl}
			x_1'&=&u'&=&\frac{u\sum_{i \in \Lambda}\gamma_ix_i}{u\sum_{i \in \Lambda} x_i}, \medskip \\
			x_j'&&&=&\frac{u\sum_{i \in \Lambda}(1-2\gamma_i)x_i}{(n-1)u\sum_{i \in \Lambda} x_i}, \text{ for any }j=2,\dots,n;
		\end{array}\right.
	\end{equation}
	with $n>1$ and $0\leq\gamma_i\leq\frac{1}{2}$ for any $i\in\Lambda$ and some non-zero. Notice that, in general, the previous operator does not map the set $S^{n,1}$ to itself, as shown in the next result.
	\begin{lemma}\label{lem:equiv_inv}
		The operator $\widetilde{V}_{\gamma_1,\dots,\gamma_n}$ defined by \eqref{norm_musminutoides} maps $S^{n,1}$ to itself if and only if $\gamma_i\neq0$ for any $i\in\Lambda$.
	\end{lemma}
	\begin{proof}
		It follows straightforward since $(\gamma_i,\frac{1-2\gamma_i}{n-1},\dots,\frac{1-2\gamma_i}{n-1},\gamma_i)\in S^{n,1}$ if and only if $\gamma_i\neq0$ for any $i\in\Lambda$.
		
	\end{proof}
Thanks to the following result, it is not necessary to restrict the values of  $\{\gamma_i\}_{i \in \Lambda}$ so much. Instead, it suffices to take an invariant subset of $S^{n,1}$  under  $\widetilde{V}_{\gamma_1,\dots,\gamma_n}$. Actually, we will determine the largest invariant subset. To do this, let $\widetilde{V}_{\gamma_1,\dots,\gamma_n}$ be the normalised gonosomal operator defined by \eqref{norm_musminutoides} and consider the following two subsets of $S^{n,1}$, which depend on the parameters $\gamma_1,\dots,\gamma_n$:
	\begin{align*}
		R_{\gamma_1,\dots,\gamma_n}&=\{(x_1,\ldots,x_n,u)\in S^{n,1} \colon \gamma_ix_i>0 \text{ for some } i\in \Lambda\},\\
		T_{\gamma_1,\dots,\gamma_n}&=\{(x_1,\ldots,x_n,u)\in S^{n,1} \colon   (1-2\gamma_i)x_i>0 \text{ for some } i\in \Lambda\}.
	\end{align*} For simplicity, if there is no risk of confusion, we will denote these sets as $R$ and $T$. We establish the next result.

	\begin{proposition}\label{prop:inv_subsets}
		In the setting above, if $\gamma_1 > 0$, then $R$ is the biggest invariant with respect to $\widetilde{V}_{\gamma_1,\dots,\gamma_n}$. However, if $\gamma_1 = 0$, then such subset is $R\cap T$.
	\end{proposition}
	\begin{proof}
		Let $s^{(0)}=(x_1,\ldots,x_n,u)\in S^{n,1}$.    First, notice that by the construction of the normalised gonosomal operator, $s^{(1)}$ always belongs to the simplex of $\mathbb{R}^{n+1}$.
		For the first part, suppose that $s^{(0)} \in R$.
		It is clear that $s^{(1)} \in R$ since
		\[x^{(1)}=u^{(1)}=\frac{u\sum_{i \in \Lambda}\gamma_ix_i}{u\sum_{i \in \Lambda}x_i}>0.\]
		To prove that $R$ is the biggest invariant contained in $S^{n,1}$ just consider an initial point $s^{(0)}\in S^{n,1}\setminus R$. So, $u^{(1)}=0$ and thus $s^{(1)}\notin S^{n,1}$.
		For the second part, assume that $\gamma_1=0$ and consider an initial point $s^{(0)}=(x_1,\ldots,x_n,u) \in R\cap T$. Then, we know there exist $i,j \in \Lambda$ such that $x_i\gamma_i>0$ and $(1-2\gamma_j)x_j>0$. Consequently,
		\[x^{(1)}=u^{(1)}=\frac{u\sum_{i \in \Lambda}\gamma_ix_i}{u\sum_{i \in \Lambda}x_i}>0\quad\text{and}\quad x_k^{(1)}=\frac{u\sum_{i \in \Lambda}(1-2\gamma_i)x_i}{(n-1)u\sum_{i \in \Lambda}x_i}>0\]
		for any $k=2,\dots,n$. Hence, there exist $i,j\in\Lambda$ such that $\gamma_ix_i^{(1)} >0$ and $(1-2\gamma_j)x_j^{(1)}>0$, which yields that  $s^{(1)} \in T$. Now, we have to see that $R\cap T$ is the largest invariant set contained in $S^{n,1}$. Consider $s^{(0)}\in S^{n,1}\setminus R$, so $x_i\gamma_i=0$ for any $i\in \Lambda$. Thus, $u^{(1)}=0$ and accordingly $s^{(1)}\notin S^{n,1}$. Now, we take an element 
		$s^{(0)}\in S^{n,1}\setminus T$, implying that $(1-2\gamma_i)x_i=0$ for all $i\in \Lambda$. Therefore $x_k^{(1)}=0$ for any $k=2,\dots,n$ and thus $u^{(2)}=0$ which implies that $s^{(2)}\notin S^{n,1}$. 
	\end{proof}

	Next, we study the limits of \eqref{norm_musminutoides} in cases where every non-zero fixed point of $V_{\gamma_1,\dots,\gamma_n}$ is non-negative. This happens whether $\gamma_1=\frac{1}{2}$ or $\sum_{i=2}^n\gamma_i=(n-1)\gamma_1$ by Corollary \ref{cor:GO_norm}. Notice that,  necessarily  $\gamma_1\neq0$, then, by Proposition~\ref{prop:inv_subsets}, $R$ is the largest invariant subset of $S^{n,1}$ with respect to $\widetilde{V}_{\gamma_1,\dots,\gamma_n}$. 
    Again, to carry out this study, we consider the ``normalised'' version of \eqref{gen_musminutoides_simp}, which will be given by 
	\begin{equation}\label{norm_musminutoides_simp}
		\widetilde{SV}_{n,\gamma,\lambda}:\left\{
		\begin{array}{lcl}
			x'&=&\frac{x(\gamma x+\lambda y)}{x(x+(n-1)y)}, \medskip  \\
			y'&=&\frac{x((1-2\gamma) x+(n-1-2\lambda)y)}{x(n-1)(x+(n-1)y)},
		\end{array}\right.
	\end{equation}
      with  $n\in\mathbb{N},n>1$, $0 < \gamma=\gamma_1 \leq \frac{1}{2}$ and $0 \leq \lambda=\sum_{i=2}^n\gamma_i \leq \frac{n-1}{2}$. Observe that the set
	\[
	\widetilde{S^{n,1}}:=\{(x,y)\in\mathbb{R}^2\colon x>0,y\geq0,2x+(n-1)y=1\}
	\]
	is invariant with respect to $\widetilde{SV}_{n,\gamma,\lambda}$. 
	
	\begin{theorem}\label{th:lim_musminutoides_gen}
		In the setting above, if $\gamma=\frac{1}{2}$ or $\lambda=(n-1)\gamma$, then, for any initial point $s=(x^{(0)},y^{(0)})\in\widetilde{S^{n,1}}$, we get that $\lim_{m\to\infty}\widetilde{SV}_{n,\gamma,\lambda}^m(s)$ exists. In fact, we have:
		\begin{enumerate}[\rm (i)]
			\item if $\gamma\neq\frac{1}{2}$, then $\lim_{k \to \infty} \widetilde{SV}_{n,\gamma,\lambda}^k(s)=\left(\gamma,\frac{1-2\gamma}{n-1}\right)$.\smallskip
			\item if $\gamma=\frac{1}{2}$, then we distinguish the two following cases:\smallskip
			\begin{enumerate}[\rm (a)]
				\item if $\lambda\geq\frac{n-1}{4}$, then
				$\lim_{k\to\infty} \widetilde{SV}_{n,\gamma,\lambda}^k(s)= \big(\frac{1}{2},0\big)$;\smallskip
				\item if $\lambda<\frac{n-1}{4}$, then
				\[ \lim_{k\to\infty} \widetilde{SV}_{n,\gamma,\lambda}^k(s)= \left\{
				\begin{array}{ll}
					\big(\frac{1}{2},0\big), & {\rm if} \ s=\big(\frac{1}{2},0\big), \medskip \\
					\left(\frac{2\lambda}{n-1},\frac{n-1-4\lambda}{(n-1)^2}\right), & {\rm otherwise.} 
				\end{array}
				\right.\]
			\end{enumerate}         
		\end{enumerate}
	\end{theorem}

	\begin{proof}
		Let $s=(x^{(0)},y^{(0)})\in\widetilde{S^{n,1}}$. For item (i), notice that $x'=\frac{\gamma x+(n-1)\gamma y}{x+(n-1)y}=\gamma$ and $y'=\frac{(1-2\gamma)x+(n-1)(1-2\gamma)y}{(n-1)(x+(n-1)y)}=\frac{1-2\gamma}{n-1}$, which do not depend on the initial point. The dynamical system converges to this point in the first iteration.  
		For item (ii), as $2x^{(k)}+(n-1)y^{(k)}=1$ for any $k \in \N$, it is enough to study the limits of the sequence $\{y^{(k)}\}_{k \in \N}$. We can write		
		\begin{equation}\label{sucesion}
			y^{(k+1)}=\frac{(n-1-2\lambda)y^{(k)}}{(n-1)\big(x^{(k)}+(n-1)y^{(k)}\big)}=\frac{2(n-1-2\lambda)y^{(k)}}{(n-1)\big(1+(n-1)y^{(k)}\big)}
		\end{equation}
		for any $k \in \N$. If $y^{(0)}=0$, clearly $\displaystyle\lim_{k\to\infty} y^{(k)}=0$. On the other hand, we have that:
		\begin{align}\label{ineq2}
					y^{(k)}\left\{
			\begin{array}{ll}
				<\frac{n-1-4\lambda}{(n-1)^2}&\Longleftrightarrow \ y^{(0)} < \frac{n-1-4\lambda}{(n-1)^2}, \\
				>\frac{n-1-4\lambda}{(n-1)^2} &\Longleftrightarrow \ y^{(0)} > \frac{n-1-4\lambda}{(n-1)^2}, \\
				=\frac{n-1-4\lambda}{(n-1)^2} &\Longleftrightarrow \ y^{(0)} = \frac{n-1-4\lambda}{(n-1)^2};  \\
			\end{array}
			\right.
					\end{align}
		because $y^{(k+1)}<\frac{n-1-4\lambda}{(n-1)^2} \Leftrightarrow y^{(k)}<\frac{n-1-4\lambda}{(n-1)^2}$.
		From this, and the fact that $y^{(k+1)}>y^{(k)}\Leftrightarrow y^{(k)}<\frac{n-1-4\lambda}{(n-1)^2}$, we get that
		\begin{align}\label{ineq3}
			y^{(k+1)}\left\{ 
			\begin{array}{ll}
				> y^{(k)} &\Longleftrightarrow \ y^{(0)} < \frac{n-1-4\lambda}{(n-1)^2},  \\
				< y^{(k)} &\Longleftrightarrow \ y^{(0)} > \frac{n-1-4\lambda}{(n-1)^2},  \\
				= \frac{n-1-4\lambda}{(n-1)^2}&\Longleftrightarrow \ y^{(0)} = \frac{n-1-4\lambda}{(n-1)^2}.  \\
			\end{array}
			\right. 
				\end{align}
		Therefore, $\{y^{(k)}\}_{k \in \N}$ is an increasing (resp. decreasing) and upper (resp. lower) bounded sequence, which implies that its limit always exists. Now, we note that necessarily $\lim_{k\to\infty} x^{(k)}\neq0$. Indeed, if $\lim_{k\to\infty} x^{(k)}=0$ then  $\lim_{k\to\infty}  x^{(k)}+(n-1)y^{(k)}=1$. So, $\lim_{k\to\infty} \lambda y^{(k)} =0$ a contradiction since  $2x^{(k)}+(n-1)y^{(k)}=1$ for any $k \in \N$. Taking this into account and taking limits in both sides of \eqref{sucesion}, we get that $L_y\big((n-1)^2L_y+4\lambda-(n-1)\big)=0$, that is, $L_y=0$ or $L_y=\frac{n-1-4\lambda}{(n-1)^2}$.  Observe that if $n-1=4\lambda \neq 0$, then  $L_y=0$. Now, since $\widetilde{S^{n,1}}$ is an invariant subset with respect to $\widetilde{SV}_{n,\gamma,\lambda}$, we have that $y^{(k)}\geq 0$ for all $k \in \N$. Hence $L_y\geq 0$. So, if  $\lambda >\frac{n-1}{4}$,  necessarily $L_y=0$.  For  $\lambda <\frac{n-1}{4}$, whether $y^{(0)}< \frac{n-1-4\lambda}{(n-1)^2}$ or  $y^{(0)}> \frac{n-1-4\lambda}{(n-1)^2}$ happens,  applying \eqref{ineq2} and \eqref{ineq3}, we obtain that $L_y=\frac{n-1-4\lambda}{(n-1)^2}$.
		Finally, the result follows from the fact that $x^{(k)}=\frac{1-(n-1)y^{(k)}}{2}$ for any $k\geq0$.	
	\end{proof}

    \begin{remark}\label{rem:aclaracion}
        Let $s=(x_1^{(0)},\dots,x_n^{(0)},u^{(0)})\in R$. For any $k\geq1$, it is easy to check that  
        \begin{align*}
    \widetilde{SV}_{n,\gamma,\lambda}^k(x_1^{(1)},x_2^{(1)})=(x_1^{(k+1)},x_2^{(k+1)})
       \end{align*} 
         where $x_1^{(1)}$, $x_2^{(1)}$, $x_1^{(k+1)}$ and $x_2^{(k+1)}$ are given by $\widetilde{V}_{\gamma_1,\dots,\gamma_n}(s)=(x_1^{(1)},x_2^{(1)},\dots,x_n^{(1)},u^{(1)})$  and $\widetilde{V}^{k+1}_{\gamma_1,\dots,\gamma_n}(s)=(x_1^{(k+1)},x_2^{(k+1)},\dots,x_n^{(k+1)},u^{(k+1)})$ for any $k\geq1$.
    \end{remark}
         Hence, we establish the following corollary.

	\begin{corollary}\label{cor:NGO}
		Consider the gonosomal operator $V_{\gamma_1,\dots,\gamma_n}$ given by \eqref{gen_musminutoides}, whose fixed points are non-negative and normalisable. Let $\widetilde{V}_{\gamma_1,\dots,\gamma_n}$ be its corresponding normalised gonosomal operator. Then, for any initial point $s\in R$, we get that $\displaystyle\lim_{k\to\infty}\widetilde{V}_{\gamma_1,\dots,\gamma_n}^k(s)$ exists.
	\end{corollary}
	\begin{proof}
        It follows straightforwardly from Corollary \ref{cor:GO_norm}, Theorem \ref{th:lim_musminutoides_gen} and Remark \ref{rem:aclaracion}. 
	\end{proof}
	
	Finally, we apply the previous results to study the dynamic behaviour of a concrete population.

	\begin{example}[see \cite{MABDS_03,VCCJal_09}]
		This first example models some rodent populations, such as \textit{Myopus schisticolor} (wood lemming) and \textit{Mus minutoides} (African pygmy mouse). In this case, we describe three female genotypes: $f_1\leftrightarrow \text{XX}$, $f_2\leftrightarrow \text{XX*}$ and $f_3\leftrightarrow \text{X*Y}$; and only one male genotype: $h\leftrightarrow \text{XY}$.
		Then, as explained in \cite[Example 16]{V_16}, applying Construction~\ref{const_1}, it is possible to obtain the following gonosomal algebra, which realises the results of crosses:
		\begin{align}\label{al_musminutoides}
			f_1h=\frac{1}{2}f_1+\frac{1}{2}h,\quad f_2h=\frac{1}{4}f_1+\frac{1}{4}f_2+\frac{1}{4}f_3+\frac{1}{4}h\quad\text{and}\quad f_3h=\frac{1}{2}f_2+\frac{1}{2}f_3.
		\end{align}
		Therefore, the corresponding (normalised) gonosomal operators are given by 
		\begin{equation}\label{op_musminutoides}
			V:\left\{
			\begin{array}{ll}
				x_1'=u'=\frac{1}{2}x_1u+\frac{1}{4}x_2u,\medskip \\
				x_2'=x_3'=\frac{1}{4}x_2u+\frac{1}{2}x_3u;
			\end{array}\right.
			\quad\text{and}\quad
			\widetilde{V}:\left\{
			\begin{array}{ll}
				x_1'=u'=\frac{u(2x_1+x_2)}{4u(x_1+x_2+x_3)},\medskip \\
				x_2'=x_3'=\frac{u(x_2+2x_3)}{4u(x_1+x_2+x_3)}.
			\end{array}\right.
		\end{equation}
		Notice that if $n=3$, $\gamma_1=\frac{1}{2},\gamma_2=\frac{1}{4}$ and $\gamma_3=0$, then the gonosomal algebra \eqref{al_gen_musminutoides}, the gonosomal operator \eqref{gen_musminutoides} and the normalised gonosomal operator \eqref{norm_musminutoides} correspond exactly with \eqref{al_musminutoides} and \eqref{op_musminutoides}, respectively.
		As $\gamma_1,\gamma_2\neq0$ but $\gamma_3=0$ then, by Proposition \ref{prop:inv_subsets}, we have that
		\[
		R=\{(x_1,x_2,x_3,u)\in S^{3,1}\colon x_1\neq0\text{ or }x_2\neq0\}
		\]
		is the biggest subset of $S^{3,1}$ which $\widetilde{V}$ maps to itself. Consequently,  Theorem \ref{th:lim_musminutoides_gen} and Corollary \ref{cor:NGO} can be applied. Then, since $\lambda=\gamma_2+\gamma_3=\frac{1}{4}<\frac{1}{2}=\frac{n-1}{4}$, for any initial point $s=\big(x_1^{(0)},x_2^{(0)},x_3^{(0)},u^{(0)}\big)\in R$ it holds that
		\[ \lim_{k\to\infty} \widetilde{V}^k_{\frac{1}{2},\frac{1}{4},0}(s)= \left\{
		\begin{array}{ll}
			\big(\frac{1}{2},0,0,\frac{1}{2}\big), & {\rm if} \ x_2^{(0)}=x_3^{(0)}=0, \medskip \\
			\left(\frac{1}{4},\frac{1}{4},\frac{1}{4},\frac{1}{4}\right), & {\rm otherwise.}  
		\end{array}
		\right.\]
		
	\end{example}

        \begin{interp}
            For some rodent populations, such as the \textit{Myopus schisticolor}  and \textit{Mus minutoides}, we can conclude that for any initial state $s\in R$ (the probability distribution on the set of possible genotypes$\{\text{XX},\text{XX*},\text{X*Y},\text{XY}\}$), the future of the population is always stable. If there are no XX* and X*Y individuals in the initial state, the population tends to the equilibrium state $\big(\frac{1}{2},0,0,\frac{1}{2}\big)$ where XX and XY are distributed equally. Otherwise, the population tends to $\big(\frac{1}{4},\frac{1}{4},\frac{1}{4},\frac{1}{4}\big)$, where all possible genotypes appear in the same proportion.
        \end{interp}

		\subsection{The Arctic lemming sex-determination system}\label{sec:artic_lemming}
	So far, we have only considered particular genetic examples which can be modelled by a gonosomal algebra obtained by Construction~\ref{const_1}. This subsection is devoted to studying the dynamical behaviour of a \textit{Dicrostonyx torquatus} (Artic lemming) population, which, as shown in Example \ref{ej:art_lem}, cannot be realised as the commutative duplicate of a baric algebra. Actually, we need to combine Constructions~\ref{const_1} and \ref{const_2}. Define the baric algebra $\mathcal{A}$ with basis $\{e_1,e_2,e_3\}$ and product given by $e_i^2=e_i$, $e_1e_i=\frac{1}{2}(e_1+e_i)$ and $e_2e_3=\frac{1}{2}(e_2+e_3)$ for any $i=1,2,3$. Then, we take the subspaces $F=\spa\{f_1=e_1\otimes e_1,f_2=e_1\otimes e_2,f_3=e_2\otimes e_3,f_4=e_3\otimes e_3\}$ and $M=\spa\{h=e_1\otimes e_3\}$ of $D(\mathcal{A})$. Notice that $f_4$ is not a possible genotype, but as the baric algebra $\mathcal{A}$ is defined, we need to include it to $\mu(F)\otimes\mu(M)\subset F\oplus M$. Finally, we just need to reduce the gonosomal basis $\{f_1,f_2,f_3,f_4,h\}$ by taking $I=\{4\}$ in Construction~\ref{const_2}. Therefore, the gonosomal algebra, which realises the results of crosses, has a gonosomal basis $\{f_1,f_2,f_3,h\}$ and its product is given by 
	\begin{align*}
		f_1h=\frac{1}{2}f_1+\frac{1}{2}h,\quad f_2h=\frac{1}{4}f_1+\frac{1}{4}f_2+\frac{1}{4}f_3+\frac{1}{4}h\quad\text{and}\quad
		f_3h=\frac{1}{3}f_2+\frac{1}{3}f_3+\frac{1}{3}h.
	\end{align*}
	Its associated gonosomal operator is
	\begin{equation}\label{op_artic}
		V\colon\left\{
		\begin{array}{rclcrcl}
			x_1'&=&\frac{1}{2} x_1u+\frac{1}{4}x_2u, &\quad& x_3'&=& \frac{1}{4}x_2u+\frac{1}{3}x_3u,\medskip \\
			x_2'&=&\frac{1}{4}x_2u+\frac{1}{3}x_3u, &\quad& u'&=&\frac{1}{2}x_1u+\frac{1}{4}x_2u+\frac{1}{3}x_3u;
		\end{array}
		\right.
	\end{equation}
	and its fixed points are given by the following result.
	\begin{proposition}\label{fix_points_artic}
		In the setting above, there exist two non-zero fixed points: $(2,0,0,2)$ and $(\frac{36}{25},\frac{12}{25},\frac{12}{25},\frac{12}{7})$.
	\end{proposition}
	\begin{proof}
		As $x_2^\prime=x_3^\prime$, we need to solve the system of equations given by
		\[x_1=u\left(\frac{1}{2} x_1+\frac{1}{4}x_2\right),\quad x_2=\frac{7}{12}x_2u\quad\text{and}\quad u=u\left(\frac{1}{2}x_1+\frac{7}{12}x_2\right).\] 
		First, if $x_2=0$, it is easy to check that $x_1=u=0$ or $x_1=u=2$. Otherwise, we get $u=\frac{12}{7}$. Then, from the third equation, we get that $\frac{1}{2}x_1+\frac{7}{12}x_2=1$, or equivalently, $x_1=2-\frac{7}{6}x_2$. Changing this expression of $x_1$ in the first equation, we get $x_2=\frac{12}{25}$ and, consequently, $x_1=\frac{36}{25}$.
	\end{proof}
	Next, we consider the normalised version of \eqref{op_artic}, that is,
	\begin{equation}\label{op_artic_norm}
		\widetilde{V}\colon\left\{
		\begin{array}{rclcrcl}
			x_1'&=&\frac{u(6x_1+3x_2)}{12u(x_1+x_2+x_3)}, &\quad& x_3'&=& \frac{u(3x_2+4x_3)}{12u(x_1+x_2+x_3)},\medskip \\
			x_2'&=&\frac{u(3x_2+4x_3)}{12u(x_1+x_2+x_3)}, &\quad& u'&=&\frac{u(6x_1+3x_2+4x_3)}{12u(x_1+x_2+x_3)}.
		\end{array}
		\right.
	\end{equation}
	This normalised gonosomal operator has as fixed points $(\frac{1}{2},0,0,\frac{1}{2})$ and $(\frac{7}{20},\frac{7}{60},\frac{7}{60},\frac{5}{12})$. Moreover, it is easy to check that $\widetilde{V}$ maps $S^{3,1}$ to itself.
 
	\begin{proposition}\label{prop:equiv}
		Consider the normalised gonosomal operator $\widetilde{V}$ defined by \eqref{op_artic_norm} and an initial point $s^{(0)}=(x_1^{(0)},x_2^{(0)},x_3^{(0)},u^{(0)}) \in S^{3,1}$. Then, the following assertions are equivalent:
		\begin{enumerate}[\rm (i)]
			\item $6x_1^{(0)}\leq6x_2^{(0)}+12x_3^{(0)}$ (resp. $6x_1^{(0)}\geq6x_2^{(0)}+12x_3^{(0)}$);\smallskip
			\item $x_1^{(k)}\leq3x_2^{(k)}$ (resp. $x_1^{(k)}\geq3x_2^{(k)}$) for any $k\in \N^*$;\smallskip
			\item $x_2^{(k)}\geq\frac{7}{60}$  (resp. $x_2^{(k)}\leq\frac{7}{60}$) for any $k\geq2$; and\smallskip
			\item $x_2^{(k)}\geq\frac{7}{24}-\frac{1}{2}x_1^{(k)}$ (resp. $x_2^{(k)}\leq\frac{7}{24}-\frac{1}{2}x_1^{(k)}$) for any $k\geq2$.
		\end{enumerate}
	\end{proposition}
	\begin{proof}
		Let $s^{(0)}=(x_1^{(0)},x_2^{(0)},x_3^{(0)},u^{(0)}) \in S^{3,1}$.     We only prove the result for one inequality; the other is trivially analogue.
		First, we prove that (i) is equivalent to (ii). Notice that $6x_1^{(0)}<6x_2^{(0)}+12x_3^{(0)}$ is equivalent to $x_1^{(1)}<3x_2^{(1)}$. Indeed, adding $3x_2^{(0)}$ to both sides we have that
		\begin{align*}
			6x_1^{(0)}\leq6x_2^{(0)}+12x_3^{(0)}\Longleftrightarrow 6x_1^{(0)}+3x_2^{(0)}\leq3(3x_2^{(0)}+4x_3^{(0)})
			\Longleftrightarrow x_1^{(1)}\leq3x_2^{(1)}.
		\end{align*}
		Moreover, we claim that $x_1^{(k)}\leq3x_2^{(k)}$ is equivalent to $x_1^{(k+1)}\leq3x_2^{(k+1)}$ for any $k\in \N^*$. It holds that
		\begin{align*}
			x_1^{(k)}\leq3x_2^{(k)}\Longleftrightarrow 6x_1^{(k)}\leq18x_2^{(k)}\Longleftrightarrow 6x_1^{(k)}+3x_2^{(k)}\leq3(7x_2^{(k)})\Longleftrightarrow x_1^{(k+1)}\leq3x_2^{(k+1)},
		\end{align*}
		which completes the proof.
		
		Next, we prove that both (iii) and (iv) are equivalent to (ii). Just notice that for any $k\geq1$, we have that 
		\begin{gather*}
			x_2^{(k+1)}=\frac{7x_2^{(k)}}{12(x_1^{(k)}+2x_2^{(k)})}\geq\frac{7}{60}\Longleftrightarrow x_1^{(k)}\leq3x_2^{(k)}; \\
			x_2^{(k+1)}\geq\frac{7}{24}-\frac{1}{2}x_1^{(k+1)}\Longleftrightarrow\frac{7x_2^{(k)}}{12(x_1^{(k)}+2x_2^{(k)})}\geq\frac{7}{24}-\frac{6x_1^{(k)}+3x_2^{(k)}}{24(x_1^{(k)}+2x_2^{(k)})}
			\Longleftrightarrow x_1^{(k)}\leq3x_2^{(k)}.
		\end{gather*}
			\end{proof}
    \begin{remark}\label{rem:limit_u_2}
            Let be the operator $\widetilde{V}$  defined in \eqref{op_artic_norm} and assume that $\lim_{k\to\infty}x_1^{(k)}$,  
            $\lim_{k\to\infty}x_2^{(k)}$ and $\lim_{k\to\infty}x_3^{(k)}$ exist for any initial point $s=(x_1^{(0)},x_2^{(0)},x_3^{(0)},u^{(0)}) \in S^{3,1}$. Then,  notice that  $\lim_{k\to\infty}u^{(k)}\neq0$. Indeed, by contrary,  if $\lim_{k\to\infty}u^{(k)}=0$ then $\lim_{k\to\infty}6x_1^{(k)}+3x_2^{(k)}+4x_3^{(k)}=0$. Then, it necessarily holds that $\lim_{k\to\infty}x_1^{(k)}=\lim_{k\to\infty}x_2^{(k)}=\lim_{k\to\infty}x_3^{(k)}=0$, a contradiction with the fact that $x_1^{(k)}+x_2^{(k)}+x_3^{(k)}+u^{(k)}=1$ for any $k\geq0$.
        \end{remark}
	\begin{theorem}
		Consider the normalised gonosomal operator $\widetilde{V}$ given by \eqref{op_artic_norm}. Then, for any initial point $s=(x_1^{(0)},x_2^{(0)},x_3^{(0)},u^{(0)}) \in S^{3,1}$, it holds that
		\[
		\lim_{k \to \infty} \widetilde{V}^k(s)= \left\{
		\begin{array}{ll}
			(\frac{1}{2},0,0,\frac{1}{2}), & {\rm if} \ x_2^{(0)}=x_3^{(0)}=0, \smallskip \\
			(\frac{7}{20},\frac{7}{60},\frac{7}{60},\frac{5}{12}), & {\rm otherwise.}  
		\end{array}
		\right.
		\]
	\end{theorem}
	\begin{proof}
		First, notice that if $x_2^{(0)}=x_3^{(0)}=0$, then clearly $\widetilde{V}^n(s)=(\frac{1}{2},0,0,\frac{1}{2})$ for any $n\in\mathbb{N}$ and so $\lim_{n \to \infty} \widetilde{V}^n(s)=(\frac{1}{2},0,0,\frac{1}{2})$.
		
		Otherwise, we can ensure that $x_2^{(k)}\neq0$ for any $k\in \N^*$. Moreover, after some computations, it is easy to check that $x_2^{(k+1)}\geq x_2^{(k)}$ (resp. $x_2^{(k+1)}\leq x_2^{(k)}$) for any $k\in \N^*$ if and only if $x_2^{(k)}\leq\frac{7}{24}-\frac{1}{2}x_1^{(k)}$ (resp. $x_2^{(k)}\geq\frac{7}{24}-\frac{1}{2}x_1^{(k)}$) for any $k\in \N^*$. Then, by Proposition \ref{prop:equiv}, for any $k\in \N^*$, we have that
		\begin{align*}
			x_2^{(k)}\left\{
			\begin{array}{ll}
				\geq\frac{7}{60}, & {\rm if} \ 6x_1^{(0)} \leq 6x_2^{(0)}+12x_3^{(0)}, \medskip \\
				\leq\frac{7}{60}, & {\rm if} \ 6x_1^{(0)} \geq 6x_2^{(0)}+12x_3^{(0)},
			\end{array}
			\right.
			\quad\text{and}\quad
			x_2^{(k+1)}\left\{ 
			\begin{array}{ll}
				\leq x_2^{(k)}, & {\rm if} \ 6x_1^{(0)} \leq 6x_2^{(0)}+12x_3^{(0)}, \medskip \\
				\geq x_2^{(k)}, & {\rm if} \ 6x_1^{(0)} \geq 6x_2^{(0)}+12x_3^{(0)},
			\end{array}
			\right. 
		\end{align*}
            Consequently, in both cases, we have that $\{x_2^{(k)}\}_{k\in \N^*}$ is a monotone bounded sequence; its limit exists and, moreover, is a positive number.  Furthermore, since 
            \[
            x_1^{(k)}=\frac{7 x_2^{(k)}}{12 x_2^{(k+1)}}-2x_2^{(k)}\quad\text{and}\quad u^{(k+1)}=\frac{6x_1^{(k)}+7x_2^{(k)}}{12(x_1^{(k)}+2x_2^{(k)})}
            \]
            for any $k\in \N^*$, we deduce that $\{x_1^{(k)}\}_{k\in \N^*}$ and $\{u^{(k)}\}_{k\in \N^*}$ converge. As a consequence of Remark \ref{rem:limit_u_2}, if we take limits on both sides of the expressions which define the operator \eqref{op_artic_norm}, we have that the possible limits are exactly its fixed points: $(\frac{1}{2},0,0,\frac{1}{2})$ and $(\frac{7}{20},\frac{7}{60},\frac{7}{60},\frac{5}{12})$. Hence, as the limit of $\{x_2^{(k)}\}_{k\in \N^*}$ is positive, necessarily $\lim_{k \to \infty} \widetilde{V}^k(s)=(\frac{7}{20},\frac{7}{60},\frac{7}{60},\frac{5}{12})$.
 
        	\end{proof}  
	
	\begin{interp}
		For \textit{Dicrostonyx torquatus} population, given the previous result, we can conclude that for any initial state $s\in S^{3,1}$ (the probability distribution on the set of possible genotypes $\{\text{XX},\text{XX*},\text{X*Y},\text{XY}\}$), the future of the population is always stable. If there are no XX* and X*Y individuals in the initial state, the population tends to the equilibrium state $(\frac{1}{2},0,0,\frac{1}{2})$, where XX and XY are distributed equally. Otherwise, the population tends to the equilibrium state $(\frac{7}{20},\frac{7}{60},\frac{7}{60},\frac{5}{12})$, where the first female genotype and the male genotype are the most frequent.
	\end{interp}

	\section{A combination of XY systems and ZW systems}\label{pez}
	
In this section, we describe the dynamic behaviour of some African cichlid fish populations [see \cite{MR_13,PS_12}], which has not yet been modelled.
		 These species with polygenic sex determination (see \cite{MR_13}) have a multi-locus system, where alleles at an XY locus on chromosome seven and a ZW locus on chromosome five segregate independently. Most importantly, the W allele overrides the Y male determiner, so ZWXY individuals are females. Hence, when a female with a ZW sex determiner is mated to a male with an XY sex determiner, they produce siblings with four possible sex classes: ZZXX, ZWXX and ZWXY females, and ZZXY males. That is, 
		\begin{center}
			\begin{tabular}{rcl}
				$\text{ZWXX}\times \text{ZZXY}$&$\rightarrow$&$\frac{1}{4}\text{ZZXX},\frac{1}{4}\text{ZZXY},\frac{1}{4}\text{ZWXX},\frac{1}{4}\text{ZWXY}$.
			\end{tabular}
		\end{center}
		
		As shown in \cite{PS_12}, many other genotypes and crosses with different outcomes are possible. However, we will consider a simplified version in which the only female and male genotypes are the previous ones. Moreover, to show such dominance of W over Y, we will assume that W causes the elimination of Y during gametogenesis. Hence, the remaining crosses are:
		\begin{center}
			\begin{tabular}{rcl}
				$\text{ZZXX}\times \text{ZZXY}$&$\rightarrow$&$\frac{1}{2}\text{ZZXX},\frac{1}{2}\text{ZZXY}$; \ and \\
				$\text{ZWXY}\times \text{ZZXY}$&$\rightarrow$&$\frac{1}{4}\text{ZZXX},\frac{1}{4}\text{ZZXY},\frac{1}{4}\text{ZWXX},\frac{1}{4}\text{ZWXY}$.
			\end{tabular}
		\end{center}
\noindent		
In order to build the corresponding gonosomal algebra, we consider two spaces $A$ and $B$, with bases $\{a_1,a_2\}$ and $\{b_1,b_2\}$, respectively. Then, the space $\mathcal{A}=A\otimes B$ with basis $\{ e_{(i,j)}=a_i \otimes b_j\}_{i,j \in \{1,2\}}$ and the multiplication given by 
		\[ e_{(i,j)}e_{(k,l)}= \left\{
		\begin{array}{ll}
			\frac{1}{4}\big(e_{(i,j)}+e_{(i,l)}+e_{(k,j)}+e_{(k,l)}\big), & {\rm if} \ (i,j),(k,l)\neq(2,2), \smallskip \\
			\frac{1}{2}\big(e_{(i,1)}+e_{(2,1)}\big), & {\rm otherwise;}  
		\end{array}
		\right.
		\]
		is a baric algebra. Next, we define 
        \[
		I=\spa\big\{e_{(i,j)}\otimes e_{(k,l)}-e_{(k,l)}\otimes e_{(i,j)},e_{(i,j)}\otimes e_{(k,l)}-e_{(i,l)}\otimes e_{(k,j)}\colon i,j,k,l\in \{1,2\}\big\},
		\]
        and take the subspaces $F=\spa\{ e_{(1,1)}\otimes e_{(1,1)},e_{(1,1)}\otimes e_{(2,1)},e_{(1,1)}\otimes e_{(2,2)}\}$ and $M=\spa\{ e_{(1,1)}\otimes e_{(1,2)}\}$ of the quotient $(\mathcal{A}\otimes\mathcal{A})/I$.
		So, analogously to Construction \ref{const_1}, this allows us to obtain a gonosomal algebra with gonosomal basis $\{f_1=e_{(1,1)}\otimes e_{(1,1)},f_2=e_{(1,1)}\otimes e_{(2,1)},f_3=e_{(1,1)}\otimes e_{(2,2)},h=e_{(1,1)}\otimes e_{(1,2)}\}$ and product given by
		\begin{align}\label{al_af_fish}
			f_1h=\frac{1}{2}f_1+\frac{1}{2}h\quad\text{and}\quad
			f_2h=f_3h=\frac{1}{4}f_1+\frac{1}{4}f_2+\frac{1}{4}f_3+\frac{1}{4}h.
		\end{align}
		Using the coding $a_1\leftrightarrow \text{Z}$, $a_2\leftrightarrow \text{W}$, $b_1\leftrightarrow \text{X}$ and $b_2\leftrightarrow \text{Y}$, we obtain the desired frequency distribution of crosses. Furthermore, the corresponding (normalised) gonosomal operators are given by 
		\begin{equation}\label{op_af_fish}
			V\colon\left\{
			\begin{array}{ll}
				x_1'=u'=\frac{1}{2}x_1u+\frac{1}{4}x_2u,+\frac{1}{4}x_3u,\medskip\\
				x_2'=x_3'=\frac{1}{4}x_2u+\frac{1}{4}x_3u;
			\end{array}\right.
			\quad\text{and}\quad
			\widetilde{V}\colon\left\{
			\begin{array}{ll}
				x_1'=u'=\frac{2x_1+x_2+x_3}{4(x_1+x_2+x_3)},\medskip\\
				x_2'=x_3'=\frac{x_2+x_3}{4(x_1+x_2+x_3)}.
			\end{array}\right.
		\end{equation}
		Notice that if $n=3$, $\gamma_1=\frac{1}{2}$ and $\gamma_2=\gamma_3=\frac{1}{4}$, then the gonosomal algebra \eqref{al_gen_musminutoides}, the gonosomal operator \eqref{gen_musminutoides} and the normalised gonosomal operator \eqref{norm_musminutoides} correspond exactly with \eqref{al_af_fish} and \eqref{op_af_fish}, respectively. Now, by Lemma \ref{lem:equiv_inv}, it is clear that $S^{3,1}$ is invariant with respect to $\widetilde{V}$. Moreover, as $\gamma_1=\frac{1}{2}$ then, by Corollary \ref{cor:GO_norm}, all non-zero fixed points are non-negative. Consequently, Theorem \ref{th:lim_musminutoides_gen} and Corollary \ref{cor:NGO} can be applied. Then, since $\lambda=\gamma_2+\gamma_3=\frac{1}{2}\geq\frac{1}{2}=\frac{n-1}{4}$, for any initial point $s=\big(x_1^{(0)},x_2^{(0)},x_3^{(0)},u^{(0)}\big)\in S^{3,1}$ it holds that
		\[ 
		\lim_{k\to\infty} \widetilde{V}_{\frac{1}{2},\frac{1}{4},0}^k(s)=\left(\frac{1}{2},0,0,\frac{1}{2}\right).
		\]

	\begin{interp}
		For the African cichlid fish population, we can conclude that for any initial state $s\in S^{3,1}$ (the probability distribution on the set of possible genotypes $\{\text{ZZXX},\text{ZWXX},\text{ZWXY},\text{ZZXY}\}$), the future of the population always tends to the equilibrium state $\big(\frac{1}{2},0,0,\frac{1}{2}\big)$, that is, ZZXX and ZZXY individuals will survive in the same proportion, but ZWXX and ZWXY individuals will disappear in the future.
	\end{interp}

    \section*{Acknowledgements}
	The first author is supported by the Agencia Estatal de Investigaci\'on (Spain) through project  PID2023-152673NB-I00 and by the Junta de Andaluc\'{\i}a  through project  FQM-336,  both  with FEDER funds. This author thanks the Universidade de Santiago de Compostela 
 for their hospitality and generosity. The second and the third authors are supported by Agencia Estatal de Investigaci\'on (Spain), grant PID2020-115155GB-I00 (European FEDER support included, UE) and
by Xunta de Galicia through the Competitive Reference Groups (GRC), ED431C
2023/31.
The third author is also supported by FPU21/05685 scholarship, Ministerio de Educaci\'on y Formaci\'on Profesional (Spain).

\section*{Declarations}

\subsection*{Ethical Approval:}

This declaration is not applicable.

\subsection*{Conflicts of interests/Competing interests:} We have no conflicts of interests/competing interests to disclose.

\subsection*{Authors' contributions:}

All authors contributed equally to this work. 

\subsection*{Data Availability Statement:} The authors confirm that the data supporting the findings of this study are available within the article.


\vskip 1cm

\begin{thebibliography}{10}

\bibitem{A_19_asymp}
{\sc A.~T. Absalamov}, {\em Asymptotical behavior of trajectories for an
  evolution operator}, Uzbek Mathematical Journal, 4 (2019), pp.~4--11.

\bibitem{A_20_eigen}
{\sc A.~T. Absalamov}, {\em On the eigenvalues
  of a gonosomal evolution operator.}, Uzbek Mathematical Journal, 4 (2020),
  pp.~4--10.

\bibitem{A_21_attract}
{\sc A.~T. Absalamov}, {\em The global
  attractiveness of the fixed point of a gonosomal evolution operator},  Discontinuity, Nonlinearity, and Complexity, 10
  (2021), pp.~143--149.

\bibitem{AR_20}
{\sc A.~T. Absalamov and U.~A. Rozikov}, {\em The dynamics of gonosomal
  evolution operators}, Journal of Applied Nonlinear Dynamics, 9 (2020),
  pp.~247--257.

\bibitem{B_11_popdyn}
{\sc N.~Baca\"er}, {\em A short history of mathematical population dynamics},
  Springer-Verlag London, Ltd., London, 2011.

\bibitem{CMFRB_04}
{\sc R.~Cordaux, A.~Michel-Salzat, M.~Frelon-Raimond, T.~Rigaud, and
  D.~Bouchon}, {\em Evidence for a new feminizing wolbachia strain in the
  isopod armadillidium vulgare: evolutionary implications}, Heredity, 93
  (2004), pp.~78--84.

\bibitem{Etherington_40}
{\sc I.~M.~H. Etherington}, {\em Genetic algebras}, Proceedings of the Royal
  Society of Edinburgh, 59 (1940), pp.~242--258.

\bibitem{Etherington_41_Dup}
{\sc I.~M.~H. Etherington}, {\em Duplication of
  linear algebras}, Proceedings of the Edinburgh Mathematical Society, 6
  (1941), pp.~222--230.

\bibitem{Etherington_41}
{\sc I.~M.~H. Etherington}, {\em Non-associative
  algebra and the symbolism of genetics}, Proceedings of the Royal Society of
  Edinburgh, Section B: Biological Sciences, 61 (1941), pp.~24--42.

  \bibitem{LLR_14}
{\sc A.~Labra, M.~Ladra, and U.~A. Rozikov}, {\em An evolution algebra in
  population genetics}, Linear Algebra Appl., 457 (2014), pp.~348--362.

\bibitem{LR_13}
{\sc M.~Ladra and U.~A. Rozikov}, {\em Evolution algebra of a bisexual
  population}, Journal of Algebra, 378 (2013), pp.~153--172.

\bibitem{Lyubich_92}
{\sc Y.~I. Lyubich}, {\em Mathematical
  structures in population genetics}, Biomathematics, vol.~22, Springer, 1992.

\bibitem{MABDS_03}
{\sc J.~Marchal, M.~Acosta, M.~Bullejos, R.~Díaz de~la Guardia, and
  A.~Sánchez}, {\em Sex chromosomes, sex determination, and sex-linked
  sequences in microtidae}, Cytogenetic and Genome Research, 101 (2003),
  pp.~266--273.

\bibitem{MR_13}
{\sc E.~C. Moore and R.~B. Roberts}, {\em Polygenic sex determination}, Current
  Biology, 23 (2013), pp.~R510--R512.

\bibitem{PS_12}
{\sc N.~F. Parnell and J.~T. Streelman}, {\em Genetic interactions controlling
  sex and color establish the potential for sexual conflict in {L}ake {M}alawi
  cichlid fishes}, Heredity, 110 (2012), pp.~239--246.

\bibitem{Reed_97}
{\sc M.~Reed}, {\em Algebraic structure of genetic inheritance}, Bulletin of
  the American Mathematical Society, 34 (1997), pp.~107--130.

\bibitem{R_19}
{\sc U.~A. Rozikov}, {\em Population Dynamics: Algebraic and Probabilistic
  Approach}, World Scientific,  2019.

\bibitem{RSV_24}
{\sc U.~A. Rozikov, S.~Shoyimardonov, and R.~Varro}, {\em Gonosomal algebras
  and associated discrete-time dynamical systems}, Journal of Algebra, 638
  (2024), pp.~153--188.

\bibitem{RV_15}
{\sc U.~A. Rozikov and R.~Varro}, {\em Dynamical systems generated by a
  gonosomal evolution operator}, 
  Discontinuity, Nonlinearity and Complexity, 5 (2016),  pp.~173--185.

\bibitem{Tian_08}
{\sc J.~P. Tian}, {\em Evolution algebras and their applications}, vol.~1921 of
  Lecture Notes in Mathematics, Springer, Berlin, 2008.

\bibitem{TV_06}
{\sc J.~P. Tian and P.~Vojt{\v{e}}chovsk{\'y}}, {\em Mathematical concepts of
  evolution algebras in non-{M}endelian genetics}, Quasigroups Related Systems,
  14 (2006), pp.~111--122.

\bibitem{V_19_temperature}
{\sc N.~Valenzuela and V.~A. Lance}, {\em Temperature-dependent sex
  determination in vertebrates}, Smithsonian Institution Scholarly Press, 2019.

\bibitem{V_16}
{\sc R.~Varro}, {\em Gonosomal algebra}, Journal of Algebra, 447 (2016),
  pp.~1--30.

\bibitem{VCCJal_09}
{\sc F.~Veyrunes, P.~Chevret, J.~Catalan, R.~Castiglia, J.~Watson, G.~Dobigny,
  T.~J. Robinson, and J.~Britton-Davidian}, {\em A novel sex determination
  system in a close relative of the house mouse}, Proceedings of the Royal
  Society B: Biological Sciences, 277 (2009), pp.~1049--1056.

\bibitem{WB_80}
{\sc A.~W\"orz-Busekros}, {\em Algebras in genetics},  Lecture Notes
  in Biomathematics, vol.~36, Springer-Verlag, Berlin-New York, 1980.

\end{thebibliography}
\end{document}